\documentclass[11pt]{article}
\input{amssymb.sty}
\usepackage{amsmath, amsthm}

\usepackage{hyperref} 
\usepackage{nameref}
\usepackage{makeidx}
\usepackage{enumerate}
\usepackage{stmaryrd}
\usepackage{mathrsfs}
\usepackage{tikz}
\allowdisplaybreaks

\title{Maximal amenability of the generator subalgebra in $q$-Gaussian von Neumann algebras}
\author{\textsc{Sandeepan Parekh, Koichi Shimada, Chenxu Wen}}

\AtEndDocument{\bigskip{\footnotesize%
\textsc{Sandeepan Parekh}: Department of Mathematics, Vanderbilt University, 1326 Stevenson Center, Nashville, TN 37240, United States \par  
  \textit{E-mail address}: \texttt{sandeepan.parekh@vanderbilt.edu} \par
   \addvspace{\medskipamount}
 \textsc{Koichi Shimada}: Department of Mathematics, Kyoto University, Kyoto 606-8502, Japan \par  
  \textit{E-mail address}: \texttt{kshimada@math.kyoto-u.ac.jp} \par
  \addvspace{\medskipamount}
 \textsc{Chenxu Wen}: Department of Mathematics, University of California, Riverside, CA 92521, United States \par  
  \textit{E-mail address}: \texttt{chenxuw@ucr.edu} \par
}}

\date{}

\newcommand\norm[1]{\left\lVert#1\right\rVert}
\newtheorem{thm}{Theorem}
\newtheorem*{thm*}{Theorem}
\newtheorem*{cor*}{Corollary}
\newtheorem*{conj*}{Conjecture}
\newtheorem{prop}[thm]{Proposition}
\newtheorem*{prop*}{Proposition}
\newtheorem{cor}[thm]{Corollary}
\newtheorem{lem}[thm]{Lemma}
\newtheorem*{defn*}{Definition}
\newtheorem*{que*}{Question}
\theoremstyle{definition}
\newtheorem{defn}[thm]{Definition}
\newtheorem{rem}[thm]{Remark}

\def\H{{\mathcal{H}}}
\def\F{{\mathcal{F}}}
\def\G{{\Gamma}}

\newcommand{\qn}[1]{[{#1}]_q}             
\newcommand{\qc}[2]{\binom{#1}{#2}_q}     

\def \mkqPascal#1 
{
\begin{center}
	\begin{tikzpicture}[scale=1.3]
	\pgfmathsetmacro{\size}{#1}
    \def\dx{20pt}
    \def\dy{30pt}
    \node (00) at (0,0) {(0,0)};
    \foreach \d in {1,...,\size}                          
	{
		\pgfmathtruncatemacro{\dd}{\d-2}                  
		\foreach \h in {-\d, -\dd, ..., \d}               
    	{
    		\pgfmathtruncatemacro{\m}{((\h +\d)/2}        
        	\pgfmathtruncatemacro{\n}{\d - \m}
    		\node (\n\m) at (\h*\dx, -\d*\dy) {(\pgfmathint{\n}\pgfmathresult,\pgfmathint{\m}\pgfmathresult)};
    	}
	}
    \pgfmathsetmacro{\sizeminusone}{\size -1}
	\foreach \r in {0,...,\sizeminusone}
	{	 
		\pgfmathtruncatemacro{\max}{#1 -1 -\r}            
		\foreach \l in {0,...,\max}                       
    	{
    		\pgfmathtruncatemacro{\ll}{\l+1}
        	\pgfmathtruncatemacro{\rr}{\r+1}
    		\draw [->] (\l\r) -- (\ll\r);
        	\draw [->] (\l\r) -- (\l\rr) node [near start, right] {$q^{\l}$};
    	}
	}
	\end{tikzpicture}
\end{center}
}

\begin{document}

\maketitle

\begin{abstract}
In this article, we give explicit examples of maximal amenable subalgebras of the $q$-Gaussian algebras, namely,
the generator masa is maximal amenable inside the $q$-Gaussian algebras for real numbers $q$
with its absolute value sufficiently small. To achieve this,
we construct a Riesz basis in
the spirit of R{\u a}dulescu \cite{radulescu91radialmasa} and develop a structural theorem for the $q$-Gaussian algebras.
\end{abstract}

\section*{Introduction}
Starting with a real Hilbert space $\H_{\mathbb{R}}$, Voiculescu's free probability theory associates each vector $\xi\in \H_{\mathbb{R}}$ with a semi-circular element $s(\xi)$, which is the sum of the left creation operator and the annihilation operator of $\xi$ on the full Fock space of $\H_{\mathbb{R}}$. The von Neumann algebra generated by $\{s(\xi): \xi\in \H_{\mathbb{R}}\}$ turns out to be isomorphic to the free group factor $L(\mathbb{F}_{\dim(\H_{\mathbb{R}})})$, which is one of the central objects of the study of II$_1$ factors (see \cite{voiculescu92book} for details). After Voiculescu's free Gaussian functor, three generalizations were developed: the first is Bo{\. z}ejko and Speicher's $q$-Gaussian functor for $-1< q <1$ in  \cite{BozejkoSpeicher91brownian}, the second is Shlyakhtenko's approach via quasi-free states in \cite{shlyakhtenko97quasi-free}, and the third is Hiai's construction of $q$-deformed Araki-Woods algebras in \cite{hiai01q-arakiwoods} which is the combination of the first two. These constructions provide von Neumann algebras which can be seen as “deformed” free group factors. Hence it is natural to ask whether they have similar
properties to those of the free group factors and whether they are isomorphic to the free
group factors. These investigations will lead us to more profound understanding of
the free group factors; they will help us understand what presentations they admit and
what properties distinguish the free group factors from other factors.

In this paper, we will focus on the study of the structure of  Bo{\. z}ejko and Speicher's $q$-Gaussian von Neumann algebras. These von Neumann algebras are currently under intense study by various authors, and many properties are known about them: assuming $\dim \H_{\mathbb{R}}\geq 2$, the $q$-Gaussian algebras are II$_1$ factors \cite{ricard2005factoriality}, non-injective \cite{BozejkoSpeicher94CPmaps, nou04non-injectivity}, strongly solid \cite{avsec2011s-solid}. Moreover, when $\dim \H_{\mathbb{R}}<\infty$ and when $|q|$ is small (which depends on the dimension) the $q$-Gaussian algebras are isomorphic to free group factors \cite{dabrowki14freePDE, GuionnetShlyakhtenko14freetransport}.

One notable aspect of a free group factor is that it contains a distinguished subalgebra, that is, the von Neumann subalgebra generated by one of the generators of the free group. It has many interesting properties: singularity, mixing property and
maximal amenability. Since $q$-Gaussians are expected to have similar properties to those of the free group factors, it is natural to try to find a $q$-Gaussian counterpart of this subalgebra. In this context, we consider the generator subalgebra in a $q$-Gaussian algebra which is defined as the von Neumann subalgebra generated by $s(\xi)$, where $\xi\in \H_{\mathbb{R}}$ is an arbitrary fixed non-zero vector. This subalgebra coincides with the original generator subalgebra in free group factors as in $\cite{popa83maxinjective}$ when $q=0$. We also note that this subalgebra is already known to be quite useful: the factoriality of the $q$-Gaussian algebras when $\dim(\H_{\mathbb{R}})\geq 2$ relies essentially on the fact that the generator subalgebra is a maximal abelian subalgebra (masa) as in \cite{ricard2005factoriality}.  Furthermore, recently it was shown that the generator subalgebra is singular \cite{BikramMukherjee16generatormasa, wen2016singularityinq-gaussian}. 

Another motivation to study the generator subalgebra comes from Peterson's conjecture about the structure of the free group factors. Recall that in \cite{op10uniquecartan} Ozawa and Popa showed the ground-breaking result that any free group factor $L(\mathbb{F}_n)$ is strongly solid, that is, if $B\subset L(\mathbb{F}_n)$ is a diffuse amenable subalgebra, then the normalizer of $B$ inside $L(\mathbb{F}_n)$ again generates an amenable subalgebra. Strong solidity remains the deepest understanding of free group factors so far, implying other deep indecomposability results about the free group factors, such as lack of Cartan subalgebras due to Voiculescu \cite{freeentropyIII}, primeness due to Ge \cite{primenessforfreegroupfactors} and solidity due to Ozawa \cite{ozawa04solidity}. One possible generalization of strong solidity is proposed by Peterson (see the end of \cite{petersonthom11cocycle}):
\begin{conj*}
Any diffuse amenable subalgebra of a free group factor $L(\mathbb{F}_n)$ has a unique maximal amenable extension inside $L(\mathbb{F}_n)$.
\end{conj*}
Evidences of Peterson's conjecture can be found in \cite{houdayer14gammastability},\cite{wen15radialmasa},\cite{BrothierWen15cup}. In particular, any generator subalgebra in a free group factor is the \textit{unique} maximal amenable extension of any of its diffuse subalgebras.

In this paper, we prove that the generator subalgebra is in fact maximal amenable inside $q$-Gaussian algebras, when the absolute value of $q$ is sufficiently small. Moreover, the result does not depend on the dimension of the real Hilbert space that we start with. In fact, our result confirms Peterson's conjecture for any diffuse subalgebra of the generator subalgebra:

\begin{thm*}
Let $\H_{\mathbb{R}}$ be a real Hilbert space. Let $q$ be any real number with $|q|$ sufficiently small and  let $A$ be a generator subalgebra of the $q$-Gaussian von Neumann algebra $M$ associated with $\H_{\mathbb{R}}$. Then $A$ is the unique maximal amenable extension of any of its diffuse subalgebras.
\end{thm*}
This result can be seen as an evidence for the generator masa being the proper
$q$-Gaussian counterpart of the generator masa of the free group factors. It is also
notable that when $\dim \H_{\mathbb{R}}=\infty$ and $|q|\neq 0$ is small, this gives the first examples of maximal
amenable masas of the $q$-Gaussian algebras. As shown in \cite{GuionnetShlyakhtenko14freetransport}, when $\dim \H_{\mathbb{R}}<\infty$ and $|q|\neq 0$ is small, $\G_q(\H_{\mathbb{R}})$ is isomorphic to $L(\mathbb{F}_{\dim \H_{\mathbb{R}}})$ . Even in this case, our main theorem is meaningful, namely, this possibly gives a family of examples supporting Peterson's conjecture which differs from the ones given in \cite{houdayer14gammastability},\cite{wen15radialmasa},\cite{BrothierWen15cup}.

The proof relies on the notion of the \textit{asymptotic orthogonality property} which is first introduced by Popa, \cite{popa83maxinjective} and used by many authors with great success (see for instance  \cite{arnaud14maxinjective},\cite{CFRW2010radialmasa},\cite{boutonnet13hyperbolic},\cite{houdayer14exoticmaxinjective},\cite{houdayer14structurebogoljubov},\cite{wen15radialmasa},\cite{BrothierWen15cup} and the references therein). In this strategy, the central theme is to show this property of the subalgebra. Basically, in order to show this property,
one expands a vector of the standard Hilbert space along an appropriate basis
and carries out some analyses. Historically, in order to do this, Popa \cite{popa83maxinjective} used the canonical basis coming from group elements and Houdayer \cite{houdayer14exoticmaxinjective} used the basis consisting of the basic tensors of the Fock space.

One may think that our situation is similar to the free-group-factor case and
one may be tempted to simply mimic the proof for that case. However, a direct imitation does not resolve the problem: if we use the usual basis
consisting of basic tensors, it is quite difficult to show this property however small
the absolute value of $q$ is. The problem is that as elements of the $q$-Fock space
get into higher tensor parts, it becomes very hard to control the norms. As a result,
the situation is completely different from that of the free-group-factor case and it
is hopeless to carry out any analysis.

To avoid this difficulty, we have to construct a suitable basis. Our key
idea comes from R{\u a}dulescu's work in \cite{radulescu91radialmasa}. There is a reason why
we look at his construction; in \cite{CFRW2010radialmasa}, Cameron-Fang-Ravichandran-White showed
the maximal amenability of a subalgebra which is called the \textit{radial masa} of the
free group factors using Radulescu's basis (Later, Wen simplified the proof in \cite{wen15radialmasa}).
Although our situation is seemingly different from their situation, when the absolute
value of $q$ is small enough, it is possible to construct a basis suitable for our purpose.

By using this basis, we would like to develop a similar strategy to that of Wen
\cite{wen15radialmasa}. Unfortunately, even after the basis is constructed, the subsequent computations
are arduous; some key facts when $q = 0$ are no longer true when $q$ is non-zero. Nevertheless, we expect that what are true in the case of $q = 0$ should be ``approximatly true'' for the general $q$-Fock spaces, at least when $|q|$ is small. In this spirit, we are able to carry out the analyses.

The paper consists of 5 sections and an appendix. Section 1 is about preliminaries for the $q$-Fock space and there we establish some notations which will be used throughout the paper. In Section 2 we develop some $q$-combinatorics for later use. The next two sections will take up the majority of our paper: Section 3 introduces and proves the R{\u a}dulescu type of Rietz basis and Section 4 contains the key estimates for elements in the relative commutant of $A$ in the ultraproduct. In the last section, we establish the strong asymptotic orthogonality property for the inclusion of the generator subalgebra $A\subset M$ and finishes the proof of the main theorem. One can also recover the fullness of those $q$-Gaussian algebras using the basis we constructed.  Finally, in the appendix, we include an observation that the strong asymptotic orthogonality property implies singularity for diffuse abelian subalgebras. Since we will use the singularity of the generator masa in the proof of the strong asymptotic orthogonality property, this result does not make our argument easier.
However, this result will save us the trouble of showing the singularity once we know
the strong asymptotic orthogonality property. Hence we include it in the end of the paper.

\section*{Acknowledgement}
The majority of the work was completed during the authors' stay for the Hausdorff Trimester Program on von Neumann Algebras, 2016. The authors are grateful for the hospitality from the University of Bonn and the Hausdorff Research Institute for Mathematics. They would also thank Jesse Peterson for his encouragement.

\section{Preliminaries}
Throughout the paper we assume that $-1<q<1$ and let $\H_{\mathbb{R}}$ be a real Hilbert space with $\dim\H_{\mathbb{R}}\geq 2$. Denote by $\H:=\H_{\mathbb{R}}\otimes_{\mathbb{R}} \mathbb{C}$ the complexification of $\H_{\mathbb{R}}$. Define an inner product on $\bigoplus_{n\geq 0}\H^{\otimes n}$ by 
\[\left\langle e_1 \otimes \cdots\otimes e_n,f_1\otimes\cdots\otimes f_m\right\rangle_q=\delta_n(m)\sum_{\sigma\in S_m} q^{|\sigma|} \left\langle e_1\otimes\cdots\otimes e_n,f_{\sigma(1)} \otimes\cdots\otimes f_{\sigma(m)}\right\rangle, \]
where $S_m$ is the group of permutations on $\{1,\cdots, m\}$,  $\H^{\otimes 0}=\mathbb{C}\Omega$ is the space spanned by the vacuum vector $\Omega$, the inner product on the right-hand side is the usual one on $\H^{\otimes m}$ and by $|\sigma|$ we mean the number of inversions of $\sigma \in S_m$ given by
\[|\sigma|=\#\{(i,j)\in\{1,\cdots, m\}^2:i<j,\sigma(i)>\sigma(j)\}. \]
 The \textit{$q$-deformed Fock space} $\F_q(\H_{\mathbb{R}})$ is the completion of $(\bigoplus_{n\geq 0}\H^{\otimes n},\left\langle\cdot,\cdot \right\rangle_q)$ and $\|\cdot\|_q$ is the norm induced by this inner product. For simplicity, sometimes we will suppress the subscript $q$ for the inner product and the norm on the $q$-Fock space.

For $e\in \H_{\mathbb{R}}$, we define the \textit{left creation operator} $c(e)$ and the \textit{right creation operator} $c_r(e) $ on $\F_q(\H_{\mathbb{R}})$ by $c(e)(\Omega)=e=c_r(e)(\Omega)$ and
\begin{equation}\label{defn of creation}
\begin{split}
c(e)(e_1\otimes\cdots\otimes e_n)=e\otimes e_1\otimes \cdots\otimes e_n,\\
c_r(e)(e_1\otimes\cdots\otimes e_n)= e_1\otimes \cdots\otimes e_n\otimes e,
\end{split}
\end{equation}
for $n\geq 1$.  Both $c(e)$ and $c_r(e)$ are  bounded operators \cite[Lemma 4]{BozejkoSpeicher91brownian} and their adjoints $a(e)=c^*(e), a_r(e)=c_r^*(e)$ are called the \textit{left annihilation operator} and the \textit{right annihilation operator}, respectively, which are given by $a(e)(\Omega)=0=a_r(e)(\Omega)$ and 
\begin{equation}\label{defn of annihilation}
\begin{split}
a(e)(e_1\otimes\cdots\otimes e_n)=\sum_{1\leq i\leq n}q^{(i-1)}\left\langle e,e_i \right\rangle e_1\otimes \cdots\otimes \hat{e}_i\otimes \cdots\otimes e_n,\\
a_r(e)(e_1\otimes\cdots\otimes e_n)=\sum_{1\leq i\leq n}q^{(n-i)}\left\langle e,e_i \right\rangle e_1\otimes \cdots\otimes \hat{e}_i\otimes \cdots\otimes e_n,
\end{split}
\end{equation}
for $n\geq 1$, where $\hat{e}_i$ means a removed letter. Note that $c(e)$ and $c_r(f)$ commute but $c(e)$ and $a_r(f)$ do not commute in general.

The operators $c(e), c_r(e)$ satisfy the following important \textit{$q$-commutatiton relations} \cite{BozejkoSpeicher91brownian}:
\begin{equation}\label{commutation relation}
\begin{split}
a(e)c(f)-q c(f)a(e)=\left\langle e,f\right\rangle Id,\\
a_r(e)c_r(f)-q c_r(f)a_r(e)=\left\langle e,f\right\rangle Id.
\end{split}
\end{equation}

For $e\in \H_{\mathbb{R}}$, let 
\[s(e)=c(e)+a(e)\]
and let $\G_q(\H_{\mathbb{R}})$ be the von Neumann algebra generated by $\{s(e):e\in \H_{\mathbb{R}}\}$. We call it the \textit{$q$-Gaussian algebra associated with $\H_{\mathbb{R}}$}. It is known that $\G_q(\H_{\mathbb{R}})$ is a type II$_1$ factor \cite[Corollary 1]{ricard2005factoriality} and $\Omega$ is a generating and cyclic vector which gives the trace for $\G_q(\H_{\mathbb{R}})$ \cite[Theorem 4.3, 4.4]{BozejkoSpeicher94CPmaps}. Consequently, each element $x\in \G_q(\H_{\mathbb{R}})$ is uniquely determined by $\xi=x\cdot\Omega\in \F_q(\H_{\mathbb{R}})$ and we write $x=s(\xi)$. This notation is consistent with the above definition for $s(e), e\in \H_{\mathbb{R}}$. 

One can also define $s_r(e)=c_r(e)+a_r(e)$ for $e\in \H_{\mathbb{R}}$ and define $\G_{q,r}(\H_{\mathbb{R}}):=\{s_r(e):e\in \H_{\mathbb{R}}\}''$. Then similar to the group von Neaumann algebra case we have $\G_{q,r}(\H_{\mathbb{R}})=\G_q(\H_{\mathbb{R}})'$.

Here we record two facts that will be used in this paper. 
\begin{itemize}
\item Let $e\in \H$ be a unit vector, then
\begin{equation}\label{l^2 norm estimate}
\|e^{\otimes n}\|_q^2=[n]_q!,
\end{equation}
where $[k]_q=\dfrac{1-q^k}{1-q}$ and $[n]_q!=[1]_q\cdots [n]_q$. We also define $[0]_q!:=1$.
\item (\textbf{Wick formula},\cite[Proposition 2.7]{BozejkoKummererSpeicher97q-gaussian}) Let $e_1\otimes \cdots\otimes e_n\in \H^{\otimes n}$, then 
\begin{equation}\label{Wick formula}
\begin{split}
s(e_1\otimes \cdots\otimes e_n)&=\sum_{i=0}^{n}\sum _{\sigma\in S_n/(S_{n-i}\times S_i)} q^{|\sigma|}c(e_{\sigma(1)})\cdots c(e_{\sigma(n-i)})\\
&\times a(e_{\sigma(n-i+1)})\cdots a(e_{\sigma(n)}),
\end{split}
\end{equation}
where $\sigma$ is the representative of $S_{n-i}\times S_{i}$ in $S_n$ with minimal number of inversions.

\end{itemize}

From now on we fix a unit vector $e\in \H_{\mathbb{R}}$ and we call the von Neumann subalgebra $\G_q(\mathbb{R}e)\subset \G_q(\H_{\mathbb{R}})$ a \textit{generator subalgebra}. It is shown by Ricard in \cite{ricard2005factoriality} that this gives a maximal abelian subalgebra (masa) of $\G_q(\H_{\mathbb{R}})$.

Let $T:\H_{\mathbb{R}}\rightarrow\H_{\mathbb{R}}$ be a $\mathbb{R}$-linear contraction. We still denote by $T$ its complexification given by $T(\xi+i\eta)=T(\xi)+iT(\eta), \forall \xi,\eta\in \H_{\mathbb{R}}$. Then the \textit{first quantization} $\F_q(T)$, is the bounded operator on $\F_q(\H_{\mathbb{R}})$ defined by 
\[\F_q(T)=Id_{\mathbb{C}\Omega}\oplus\bigoplus_{n\geq 1}T^{\otimes n}.\]

The \textit{second quantization} of $T$, is the unique unital completely positive map $\G_q(T)$ on $\G_q(\H_{\mathbb{R}})$ defined as
\[\G_q(T)(s(\xi))=s(\F_q(T)(\xi)).\]
In particular, if $T=E_e:\H_{\mathbb{R}}\rightarrow \mathbb{R}e$ is the orthogonal projection, then $\F_q(E_e)$ is the conditional expectation of $\G_q(\H_{\mathbb{R}})$ onto $\G_q(\mathbb{R}e)$.

To simplify notations, from now on we will write $A:=\G_q(\mathbb{R}e)$ for the generator subalgebra and $M:=\G_q(\H_{\mathbb{R}})$ for the $q$-Gaussian algebra.

\section{Some $q$-combinatorics}

In this section we will develop some formulas about combinatorics in $q$-Gaussians that will be needed in later sections.

For $n,m\in \mathbb{N}\cup \{0\}, n\geq m$, set
\[\binom{n}{m}_q =\dfrac{[n]_q!}{[m]_q!\cdot [n-m]_q!}=\prod_{i=1}^{n-m}\dfrac{1-q^{m+i}}{1-q^i}.\] 
We make the following convention 
\[\binom{n}{m}_q=0 \quad \text{whenever $m>n$ or $m<0$}.\]

The following $q$-analogue of the Pascal's identity for these $q$-binomial coefficients (c.f.  \cite[Proposition 1.8]{BozejkoKummererSpeicher97q-gaussian}) can be easily checked:
\begin{lem} For all $m\in \mathbb{Z}$ and $n\geq 0$,
\begin{equation}\label{q-binomial relation}
\binom{n+1}{m}_q=q^m\binom{n}{m}_q+\binom{n}{m-1}_q=\binom{n}{m}_q+q^{n-m+1}\binom{n}{m-1}_q.
\end{equation}
\end{lem}

Continuing the analogy, the $q$-binomial coefficients $\qc{n}{m}$ can also be seen to count `number' of weighted paths in the `$q$-Pascal's triangle' from $(0,0)$ to $(n-m,m)$. The $q$-Pascal triangle is formed from the ordinary Pascal triangle by putting a weight of $q^i$ on each (right) edge from $(i,j)$ to $(i,j+1)$ , as shown below. All the other (left) edges will have weight $1$. The weight of a path is the product of the weights on its constituent edges.  

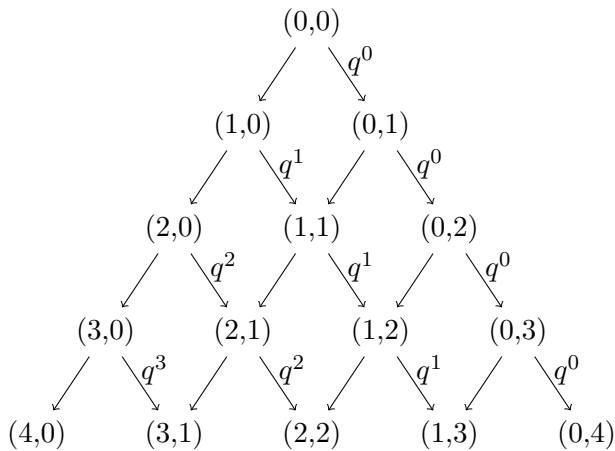
\mkqPascal{4}

For instance, the sum of all weighted paths from $(0,0)$ to $(1,2)$ is $1 + q + q\cdot q = \qn{3} = \qc{3}{2}$. It is clear from the diagram that they satisfy the second recurrence relation mentioned above, with the other one following from the symmetry of the $q$-binomial coefficient. 

\begin{lem}
For $n_1,n_2,m\in \mathbb{N}\cup\{0\}$ with $n_1+n_2\geq m$, we have
\begin{equation}\label{a complex q-binomial relation}
\sum_{i=0}^mq^{(n_1-i)(m-i)}\binom{n_1}{i}_q\binom{n_2}{m-i}_q=\binom{n_1+n_2}{m}_q.
\end{equation}
\end{lem}
\begin{proof}
Any path from $(0,0)$ to $(n_1+n_2-m, m)$ will pass through $(n_1-i,i)$ for some $0 \vee m-n_2 \leq i \leq m \wedge n_1$. This range of index corresponds exactly to the terms with non-zero contribution in the above sum. Now $\qc{n_1}{i}$ counts the sum of weighted paths from $(0,0)$ to $(n_1-i,i)$. To go from $(n_1-i,i)$ to $(n_1+n_2-m,m)$ involves travelling along paths counted by $\qc{n_2}{m-i}$. 

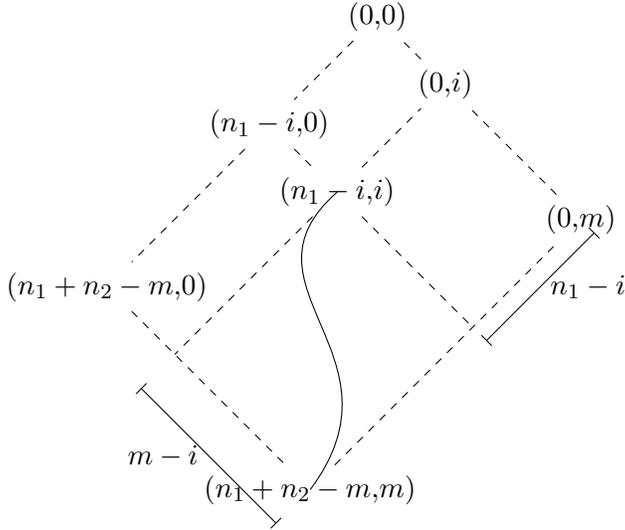
\begin{tikzpicture}[scale=0.9]
\node (ini) at (0,0) {(0,0)};                            
\node (l1) at (-1.6,-1.6) {($n_1-i$,0)};                     
\node (l2) at (-4,-4) {($n_1+n_2-m$,0)};                       
\node (mid) at (-0.6,-2.6) {($n_1-i$,$i$)};                    
\node (r1) at (1,-1) {(0,$i$)};                            
\node (r2) at (3,-3) {(0,$m$)};
\node (fin) at (-1,-7) {($n_1+n_2-m$,$m$)};
\draw [dashed] (ini) -- (l1) -- (mid) -- (r1) -- (ini);
\draw [dashed] (l1) -- (l2) -- (fin) -- (r2) -- (r1);
\draw [dashed] (mid) -- (-3,-5);
\draw [dashed] (mid) -- (1.4,-4.6);
\draw [|-|] (1.6,-4.8) -- (3.2,-3.2) node [midway, right] {$n_1-i$};
\draw [|-|] (-3.5, -5.5) -- (-1.5, -7.5) node [midway, left] {$m-i$};
\draw (-0.6,-2.6) .. controls (-2.2,-4) and (0.5,-5).. (-1,-7);
\end{tikzpicture}

However now every right edge in such a path has its weight multiplied by an extra factor of $q^{n_1-i}$. Moreover there are exactly $m-i$ right edges in such paths. The result now follows. 
\end{proof}
Let $X,Y,Z,W$ be indeterminants satisfying the following relations:
\begin{equation}\label{simple commutation of X,Y,Z,W}
\begin{split}
XY=qYX+1, XZ=ZX+W, WZ=qZW, XW=qWX.
\end{split}
\end{equation}
For convenience, we will make the convention that $x^0=1$, for all $x\in \{X,Y,Z,W\}.$
\begin{rem}
As we will see, in the following we will sometimes use negative powers of the indeterminants, however their coefficients will all be zero.
\end{rem}

We first discuss the commutation relations between powers of $X$ and $Y$.
\begin{lem}
For $n,m\in \mathbb{N}$, we have 
\begin{equation}\label{commutation relation of a(e) and c(e)}
\begin{split}
X^mY^n=\sum_{i=0}^{m}q^{(n-i)(m-i)}[i]_q!\binom{n}{i}_q\binom{m}{i}_qY^{n-i}X^{m-i}.
\end{split}
\end{equation}
\end{lem}

\begin{proof}
We prove it by induction on $m$ and $n$. 

An easy induction shows that 
\begin{equation*}
\begin{split}
XY^n=q^nY^nX+[n]_qY^{n-1},\\
X^mY=q^mYX^m+[m]_qX^{m-1},
\end{split}
\end{equation*}
which are special cases of (\ref{commutation relation of a(e) and c(e)}).

Now suppose that (\ref{commutation relation of a(e) and c(e)}) holds up to $m,n.$ Then
\begin{equation*}
\begin{split}
X^{m+1}Y^n&=\sum_{i=0}^{m}q^{(n-i)(m-i)}[i]_q!\binom{n}{i}_q\binom{m}{i}_qXY^{n-i}X^{m-i}\\
&=\sum_{i=0}^{m}\left(q^{(n-i)(m-i+1)}[i]_q!\binom{n}{i}_q\binom{m}{i}_qY^{n-i}X^{m+1-i} \right.\\
&\left. +q^{(n-i)(m-i)}[i]_q!\binom{n}{i}_q\binom{m}{i}_q[n-i]_qY^{n-1-i}X^{m-i}\right)\\
&=\sum_{i=0}^{m+1}Y^{n-i}X^{m+1-i}\left(q^{(n-i)(m+1-i)}[i]_q!\binom{n}{i}_q\binom{m}{i}_q \right.\\
& \left. +q^{(n-i+1)(m+1-i)}[i-1]_q!\binom{n}{i-1}_q\binom{m}{i-1}_q[n+1-i]_q\right)\\
&=\sum_{i=0}^{m+1}Y^{n-i}X^{m+1-i}q^{(n-i)(m+1-i)}[i]_q!\binom{n}{i}_q\left(\binom{m}{i}_q+q^{m+1-i}\binom{m}{i-1}_q\right)\\
&=\sum_{i=0}^{m+1}Y^{n-i}X^{m+1-i}q^{(n-i)(m+1-i)}[i]_q!\binom{n}{i}_q\binom{m+1}{i}_q.
\end{split}
\end{equation*}
Here in the second equation we used (\ref{simple commutation of X,Y,Z,W}), the third equality is due to the simple fact 
\[[i-1]_q!\binom{n}{i-1}_q[n-i+1]_q=[i]_q!\binom{n}{i}_q,\]
and the last equality comes from (\ref{q-binomial relation}).

The case for $X^mY^{n+1}$ is completely similar so we omit the details.

\end{proof}

Now we turn to the relation between powers of $X$ and $Z$. Naturally, $W$ also comes into play.

\begin{lem}
For $m,n\in \mathbb{N}$, we have 
\begin{equation}\label{commutation relation of a(e) and c_r(e)}
X^mZ^n=\sum_{i=0}^{m}[i]_q!\binom{n}{i}_q\binom{m}{i}_qZ^{n-i}W^iX^{m-i}.
\end{equation}
\end{lem}
\begin{proof}
Let's proceed by induction on $m$ and $n$.

One can easily show that 
\begin{equation*}
\begin{split}
XZ^n=Z^nX+[n]_qZ^{n-1}W,\\
X^mZ=ZX^m+[m]_qWX^{m-1},
\end{split}
\end{equation*}
which are special cases for (\ref{commutation relation of a(e) and c_r(e)}).

Suppose that (\ref{commutation relation of a(e) and c_r(e)}) holds up to $m,n$. Then
\begin{equation*}
\begin{split}
X^{m+1}Z^n&=\sum_{i=0}^{m}[i]_q!\binom{n}{i}_q\binom{m}{i}_qXZ^{n-i}W^iX^{m-i}\\
&=\sum_{i=0}^{m}[i]_q!\binom{n}{i}_q\binom{m}{i}_q\left(Z^{n-i}X+[n-i]_qZ^{n-i-1}W\right)W^iX^{m-i}\\
&=\sum_{i=0}^{m}[i]_q!\binom{n}{i}_q\binom{m}{i}_q\left(q^iZ^{n-i}W^iX^{m+1-i}+[n-i]_qZ^{n-i-1}W^{i+1}X^{m-i}\right)\\
&=\sum_{i=0}^{m+1}Z^{n-i}W^iX^{m+1-i}\left([i]_q!\binom{n}{i}_q\binom{m}{i}_q q^i+[i-1]_q!\binom{n}{i-1}_q\binom{m}{i-1}_q[n-i+1]_q\right)\\
&=\sum_{i=0}^{m+1}[i]_q!\binom{n}{i}_q\binom{m+1}{i}_qZ^{n-i}W^iX^{m+1-i}.
\end{split}
\end{equation*}
The other case is completely similar.
\end{proof}

\begin{prop}
Let $X=a(e), Y=c(e), Z=c_r(e)$ and $W$ be the bounded operator on $\F_q(\H_{\mathbb{R}})$ defined by
\[W(\xi)=q^{n}\xi, \forall\xi \in \H^{\otimes n}.\]
Then $X,Y,Z,W$ satisfy the relations listed in (\ref{simple commutation of X,Y,Z,W}). Consequently, (\ref{commutation relation of a(e) and c(e)}) and (\ref{commutation relation of a(e) and c_r(e)}) hold true.
\end{prop}
\begin{proof}
(\ref{simple commutation of X,Y,Z,W}) can be checked by direct computations hence (\ref{commutation relation of a(e) and c(e)}) and (\ref{commutation relation of a(e) and c_r(e)}) follow from the previous two lemmas.
\end{proof}

\section{R{\u a}dulescu basis in $q$-Fock spaces}

In this section we construct the R{\u a}dulescu basis for $\F_q(\H_{\mathbb{R}})$, which will be the fundamental tool  to study the generator subalgebra. The construction is motivated by the original construction of R{\u a}dulescu in \cite{radulescu91radialmasa}.

For each integer $k\geq 0$, we consider the following subspace of $\H^{\otimes k}\subset \F_q(\H_{\mathbb{R}})$:
\begin{equation}\label{defn of T^k}
T^k:=\{\xi \in \H^{\otimes k}:a(e)\xi=a_r(e)\xi=0\}.
\end{equation}
It is clear that $T^0=\mathbb{C}\Omega$ and each $T^k$ is non-zero (for instance, if we choose $f\in \H_{\mathbb{R}}$ with $f\perp e$, then $f^{\otimes k}\in T^k$).

For each $\xi\in \H^{\otimes k}$ and for all $s,t \in \mathbb{N}\cup\{0\}$, define
\begin{equation}
\xi_{s,t}=e^{\otimes s}\otimes \xi\otimes e^{\otimes t}\in \H^{\otimes (k+s+t)}.
\end{equation}
We also make the convention that $\xi_{s,t}=0$ if either $s<0$ or $t<0$.

We start with a few important observations.

\begin{lem}\label{a(e) acting on words}
For $\xi\in T^k$ and $s,t\geq 0$, we have the following
\begin{equation}
\begin{split}
a(e)\xi_{s,t}=[s]_q\xi_{s-1,t}+q^{s+k}[t]_q\xi_{s,t-1},\\
a_r(e)\xi_{s,t}=q^{t+k}[s]_q\xi_{s-1,t}+[t]_q\xi_{s,t-1}.
\end{split}
\end{equation}
\end{lem}
\begin{proof}
Since the left and right annihilation operators behave similarly, we just show the first equation. This is a consequence of the following equation
\begin{equation*}
\begin{split}
a(e)(W_1\otimes W_2\otimes W_3)&=(a(e)W_1)\otimes W_2\otimes W_3+q^{|W_1|}W_1\otimes(a(e)W_2)\otimes W_3
\\&+q^{|W_1|+|W_2|}W_1\otimes W_2\otimes (a(e)W_3),
\end{split}
\end{equation*}
where $W_i,i=1,2,3$ are basic words and $|W_i|$ stands for the length. By linearity, the equation still holds even if $W_2$ is a linear combination of basic words with the same length. Thus for $\xi \in T^k$, we have 
\begin{align*}
a(e)\xi_{s,t}&=a(e)(e^{\otimes s}\otimes \xi \otimes e^{\otimes t})\\
&=(a(e)e^{\otimes s})\otimes \xi \otimes e^{\otimes t}+q^s e^{\otimes s}\otimes (a(e)\xi) \otimes e^{\otimes t}+q^{s+k}e^{\otimes s}\otimes \xi \otimes (a(e)e^{\otimes t})\\
&=[s]_q\xi_{s-1,t}+0+q^{s+k}[t]_q\xi_{s,t-1}.
\end{align*}
\end{proof}
\begin{lem}\label{inclusion relations}
For $\xi\in T^k$, we have 
\begin{equation}
\begin{split}
s(e)^ns_r(e)^m\xi \in span\{\xi_{s,t}:s,t\geq 0\}, \forall n,m\geq 0,\\
\\
\xi_{s,t}\in span\{s(e)^ns_r(e)^m\xi :n,m\geq 0\}, \forall s,t\geq 0. 
\end{split}
\end{equation}

\end{lem}
\begin{proof}
For the first inclusion,
\begin{equation}
\begin{split}
s(e)^ns_r(e)^m\xi&= \left(c(e)+a(e)\right)^n\left(c_r(e)+a_r(e)\right)^m\xi\\
&=\left(c_r(e)+a_r(e)\right)^m\left(c(e)+a(e)\right)^n\xi.
\end{split}
\end{equation}
Application of the $q$-commutation relations implies that we can write $\left(c(e)+a(e)\right)^n$  and $\left(c_r(e)+a_r(e)\right)^m$ as polynomials of the form
\begin{equation}
\left(c(e)+a(e)\right)^n=\sum_{i,j\geq 0} a_{i,j}c(e)^ia(e)^j,\\
\left(c_r(e)+a_r(e)\right)^m=\sum_{k,l\geq 0} b_{k,l}c_r(e)^ka_r(e)^l.
\end{equation}
Thus $s(e)^ns_r(e)^m\xi $ is a linear combination of 
\[c(e)^ia(e)^jc_r(e)^ka_r(e)^l\xi,\]
where $i,j,k,l\in \mathbb{N}\cup\{0\}$. By the previous lemma, all such terms are in $span\{\xi_{s,t}:s,t\geq 0\}$, which yields the first inclusion.

We prove the second inclusion by inducting on $s+t$. When $s+t=0$, the conclusion clearly holds. Suppose now that the inclusion holds for $s+t\leq N$. By Lemma \ref{a(e) acting on words} we have that
\begin{equation*}
\begin{split}
s(e)\xi_{s,t}=\xi_{s+1,t}+[s]_q\xi_{s-1,t}+q^{s+k}[t]_q\xi_{s,t-1}, \\
s_r(e)\xi_{s,t}=\xi_{s,t+1}+q^{t+k}[s]_q\xi_{s-1,t}+[t]_q\xi_{s,t-1}.
\end{split}
\end{equation*}
Hence the conclusion holds for $s+t=N+1$ as well.
\end{proof}

For $k\geq 0$, let
\begin{equation*}
\begin{split}
&Q_k: \F_q(\H_{\mathbb{R}})\rightarrow \H ^{\otimes k},\\
\end{split}
\end{equation*}
be the orthogonal projections from the $q$-Fock space onto $\H ^{\otimes k}$ and we define 
\[S^k:=\H^{\otimes k}\ominus T^k. \]
For notational convenience, we let $\H^{\otimes i}=\{0\},$ for all $i<0$.

We first characterize $S^k$ as follows:
\begin{lem}\label{charaterize S^k}
For all $k\geq 0$, $S^k=span\{Q_k\left(s(e)\eta\right),Q_k\left(s_r(e)\eta\right):\eta\in \H^{\otimes l},l< k\}.$
\end{lem}
\begin{proof}
It suffices to show that $\xi\in \H^{\otimes k}$ belongs to $T^k$ if and only if
\[\left\langle \xi,Q_k(s(e)\eta) \right\rangle_q =\left\langle \xi,Q_k(s_r(e)\eta) \right\rangle_q =0\]
for any $\eta\in \H^{\otimes l},l<k$.

To see this, notice that 
\[\left\langle \xi,Q_k(s(e)\eta) \right\rangle_q=\left\langle \xi,Q_k(c(e)\eta) \right\rangle_q=\left\langle a(e)\xi,\eta \right\rangle_q.\]
Since $a(e)\xi\in \H^{\otimes (k-1)}$, we have that $\left\langle \xi,Q_k(s(e)\eta) \right\rangle_q=0$ for any $\eta \in \H^{\otimes l}, l<k$ if and only if $a(e)\xi=0$.

Similarly, $\left\langle \xi,Q_k(s_r(e)\eta) \right\rangle_q =0$ for any $\eta \in \H^{\otimes l}, l<k$ if and only if $a_r(e)\xi=0$.
\end{proof}

\begin{lem}\label{completeness}
For all $k\geq 0,$ $\H^{\otimes k}\subset span\{s(e)^ns_r(e)^m\xi:\xi\in T^l,l\leq k,n,m\geq 0\}$.
\end{lem}
\begin{proof}
We prove it by induction. When $k=0$, the statement is clearly true. Assume that the lemma holds up to $k-1$ and let $\eta\in \H^{\otimes k}$. We may further assume that $\eta\in S^k$.

By Lemma \ref{charaterize S^k}, $\eta$ is a linear combination of $Q_k(s(e)\xi)$ and $Q_k(s_r(e)\xi), \xi \in \H^{\otimes(k-1)}$. By the induction hypothesis, each $\xi\in \H^{\otimes (k-1)}$ is a linear combination of $s(e)^ns_r(e)^m\xi', \xi'\in T^l, l\leq k-1, n,m\geq 0$. Thus $\eta$ is a linear combination of $Q_k(s(e)^ns_r(e)^m\xi')$, $\xi'\in T^l,l\leq k-1, n,m\geq 0$.

Now, by Lemma \ref{inclusion relations}, $Q_k(s(e)^ns_r(e)^m\xi')\in span\{\xi'_{r,s}: r,s\geq 0,r+s+|\xi'|=k\}$ but again by Lemma \ref{inclusion relations}, each $\xi'_{r,s}\in span\{s(e)^ns_r(e)^m\xi': n,m\geq 0\}$. Therefore, we are done.
\end{proof}

\begin{lem}
Suppose that $\xi\in T^t$ and $r,s,k\geq 0$ are non-negative integers, then we have
\begin{equation}\label{annilators acting on basis}
\begin{split}
a(e)^k\xi_{r,s}=\sum_{i+j=k,i,j\geq 0}\dfrac{[r]_q!}{[r-j]_q!}\cdot \dfrac{[s]_q!}{[s-i]_q!}\cdot\binom{k}{i}_qq^{(t+r-j)i}\xi_{r-j,s-i},\\
a_r(e)^k\xi_{r,s}=\sum_{i+j=k,i,j\geq 0}\dfrac{[r]_q!}{[r-i]_q!}\cdot \dfrac{[s]_q!}{[s-j]_q!}\cdot\binom{k}{i}_qq^{(t+s-j)i}\xi_{r-i,s-j}.
\end{split}
\end{equation}
\end{lem}

\begin{proof}
This is just an induction via direct computations so we omit the details.
\end{proof}
Now we compute the inner products between $\xi_{r,s}$.
\begin{lem}\label{basic orthogonality}
Let $\xi,\eta \in \cup_{k\geq 0}T^k$ with $\xi\perp\eta$, then for any $r,s,r',s'$ non-negative integers, we have
\begin{equation}
\left\langle \xi_{r,s},\eta_{r',s'}\right\rangle=0. 
\end{equation}
Moreover, if $r+s\neq r'+s'$, then 
\[\left\langle \xi_{r,s},\xi_{r',s'}\right\rangle=0.\]
\end{lem}
\begin{proof}
The second statement is trivial so we focus on the first. By Lemma \ref{inclusion relations} and the fact that $s(e)s_r(e)=s_r(e)s(e)$, it suffices to show that
\[s(e)^ns_r(e)^m\xi\perp \eta\]
for all $n,m$ non-negative integers. Again by Lemma \ref{inclusion relations}, it reduces to show that 
\[\xi_{r,s}\perp \eta\]
for all $r,s$ non-negative integers. This is clear by the definition of $T^k$ unless $r=s=0$, but then the assumption $\xi\perp\eta$ leads to the conclusion.
\end{proof}
\begin{prop}
Let $r,s,r',s',k$ be non-negative integers with $r+s=r'+s'$ and $\xi\in T^k$ of norm 1, then 
\begin{equation}
\left\langle\xi_{r,s},\xi_{r',s'} \right\rangle=\sum_{i=0}^{r'}q^{(r-i)(r'-i)+k(r-i)+k(r'-i)}\cdot[r']_q!\cdot[s']_q!\cdot\binom{r}{i}_q\cdot\binom{s}{r'-i}_q.
\end{equation}
\end{prop}

\begin{proof}
The proof is simply a direct but lengthy computation. However for the convenience of the readers, we include the details here. For simplicity, we write $\alpha_{n,m}^i:=[i]_q!\binom{n}{i}_q\binom{m}{i}_q$ for $n,m,i\geq 0$. Note that by our convention, $\alpha_{n,m}^i=0$ when either $i>n$ or $i>m$.

Now we compute
\begin{align*}
\left\langle \xi_{r,s},\xi_{r',s'}\right\rangle
&=\left\langle c(e)^rc_r(e)^s\xi,c(e)^{r'}c_r(e)^{s'}\xi\right\rangle \\
&=\left\langle a(e)^{r'}c(e)^rc_r(e)^s\xi,c_r(e)^{s'}\xi\right\rangle \\
&\stackrel{(\ref{commutation relation of a(e) and c(e)})}{=}\left\langle \sum_{i=0}^{r'}q^{(r-i)(r'-i)}\alpha_{r,r'}^ic(e)^{r-i}a(e)^{r'-i}c_r(e)^s\xi,c_r(e)^{s'}\xi\right\rangle \\
&\stackrel{(\ref{commutation relation of a(e) and c_r(e)})}{=}\left\langle \sum_{i=0}^{r'}q^{(r-i)(r'-i)}\alpha_{r,r'}^ic(e)^{r-i}\sum_{j=0}^{r'-i}\alpha_{r'-i,s}^jc_r(e)^{s-j}W^ja(e)^{r'-i-j}\xi,c_r(e)^{s'}\xi\right\rangle\\
&\stackrel{(*)}{=}\left\langle \sum_{i=0}^{r'}q^{(r-i)(r'-i)}\alpha_{r,r'}^ic(e)^{r-i}\alpha_{r'-i,s}^{r'-i}c_r(e)^{s-(r'-i)}W^{r'-i}\xi,c_r(e)^{s'}\xi\right\rangle \\
&=\sum_{i=0}^{r'}q^{(r-i)(r'-i)}q^{(r'-i)k}\alpha_{r,r'}^i\alpha_{r'-i,s}^{r'-i}\left\langle c_r(e)^{s-(r'-i)}\xi,a(e)^{r-i}c_r(e)^{s'}\xi\right\rangle\\
&\stackrel{(\ref{commutation relation of a(e) and c_r(e)})}{=}\sum_{i=0}^{r'}q^{(r-i)(r'-i)}q^{(r'-i)k}\alpha_{r,r'}^i\alpha_{r'-i,s}^{r'-i}\left\langle c_r(e)^{s-(r'-i)}\xi,\sum_{j=0}^{r-i}\alpha_{r-i,s'}^jc_r(e)^{s'-j}W^ja(e)^{r-i-j}\xi\right\rangle\\
&\stackrel{(*)}{=}\sum_{i=0}^{r'}q^{(r-i)(r'-i)}q^{(r'-i)k}q^{(r-i)k}\alpha_{r,r'}^i\alpha_{r'-i,s}^{r'-i}\alpha_{r-i,s'}^{r-i}\left\langle c_r(e)^{s-(r'-i)}\xi,c_r(e)^{s'-(r-i)}\xi\right\rangle\\
&=\sum_{i=0}^{r'}q^{(r-i)(r'-i)}q^{(r'-i)k}q^{(r-i)k}\alpha_{r,r'}^i\alpha_{r'-i,s}^{r'-i}\alpha_{r-i,s'}^{r-i}[s-r'+i]_q!\\
&=\sum_{i=0}^{r'}q^{(r-i)(r'-i)+k(r-i)+k(r'-i)}\cdot[r']_q!\cdot[s']_q!\cdot\binom{r}{i}_q\cdot\binom{s}{r'-i}_q,
\end{align*}
where in the equations with $(*)$ we used the fact that $a(e)\xi=0$.

\end{proof}

For later use, we define two constants depending on $q$:
\[C(q):=\prod_{i=1}^{\infty}\dfrac{1}{1-q^i}, \quad D(q):=\prod_{i=1}^{\infty}(1+|q|^i).\]
Basic calculus shows that whenever $-1<q<1$, the above two limits exist unconditionally.

We record a simple but very useful estimate here for later references.
\begin{lem}\label{estimates for binomimal coefficients}
For all $-1<q<1$ and for all $n,m\geq 0$, we have 
\begin{align*}
\left|\binom{n}{m}_q\right|^{\pm 1}\leq D(q)C(|q|).
\end{align*}
\end{lem}

\begin{lem}\label{Estimate of inner products}
Let $r,s,r',s',k$ be non-negative integers with $r+s=r'+s', r\geq r'$ and $\xi\in T^k$ of norm 1. Then for each $-1<q<1$, there are constants $E(q), F(q)$ such that 
\begin{equation}\label{formula: estimate of inner product}
E(q)|q|^{k(r-r')}[r+s]_q!\leq|\left\langle\xi_{r,s},\xi_{r',s'} \right\rangle|\leq F(q)|q|^{k(r-r')}[r+s]_q!.
\end{equation}
Moreover, we have that 
\[\lim_{q\to 0}E(q)=\lim_{q\to 0}F(q)=1.\]
\end{lem}
\begin{proof}
When $0\leq q<1$, we have 
\begin{align*}
\left\langle\xi_{r,s},\xi_{r',s'} \right\rangle
&=\sum_{i=0}^{r'}q^{(r-i)(r'-i)+k(r-i)+k(r'-i)}\cdot[r']_q!\cdot[s']_q!\cdot\binom{r}{i}_q\cdot\binom{s}{r'-i}_q\\
&\leq q^{k(r-r')}\sum_{i=0}^{r'}q^{(r-i)(r'-i)}\cdot[r']_q!\cdot[s']_q!\cdot\binom{r}{i}_q\cdot\binom{s}{r'-i}_q\\
&\stackrel{(\ref{a complex q-binomial relation})}{=}q^{k(r-r')}\cdot[r']_q!\cdot[s']_q!\cdot\binom{r+s}{r'}_q\\
&=q^{k(r-r')}[r+s]_q!,
\end{align*}
where the last equality comes from the assumption $r+s=r'+s'$. This proves the inequality on the right side.

For the inequality on the left side, simply note that 
\begin{align*}
&\sum_{i=0}^{r'}q^{(r-i)(r'-i)+k(r-i)+k(r'-i)}\cdot[r']_q!\cdot[s']_q!\cdot\binom{r}{i}_q\cdot\binom{s}{r'-i}_q\\
&\geq \sum_{i=r'}^{r'}q^{(r-i)(r'-i)+k(r-i)+k(r'-i)}\cdot[r']_q!\cdot[s']_q!\cdot\binom{r}{i}_q\cdot\binom{s}{r'-i}_q\\
&=q^{k(r-r')}[r']_q![s']_q!\cdot\binom{r}{r'}_q\\
&=q^{k(r-r')}[r']_q![s']_q!\cdot \dfrac{[r]_q!}{[r']_q![r-r']_q!}\cdot \dfrac{[r+s]_q!}{[r+s]_q!}\\
&=q^{k(r-r')}[r+s]_q! \cdot \dfrac{[s']_q![r]_q!}{[r-r']_q![r+s]_q!}\\
&= q^{k(r-r')}[r+s]_q! \cdot \dfrac{(1-q^{r-r'+1})\cdots (1-q^{r})}{(1-q^{s'+1})\cdots (1-q^{r+s})}\\
&\geq q^{k(r-r')}[r+s]_q! \cdot (1-q^{r-r'+1})\cdots (1-q^{r})\\
&\geq \dfrac{1}{C(q)}q^{k(r-r')}[r+s]_q!.
\end{align*}
Thus if we let $E(q)=\dfrac{1}{C(q)}, F(q)=1$, we are done.

Now assume $-1<q<0$. By Lemma \ref{estimates for binomimal coefficients} we have
\begin{align*}
[r']_q![s']_q!\leq [r'+s']_q!D(q)C(|q|).
\end{align*}
Therefore,
\begin{align*}
&\left|\sum_{i=0}^{r'}q^{(r-i)(r'-i)+k(r-i)+k(r'-i)}\cdot[r']_q!\cdot[s']_q!\cdot\binom{r}{i}_q\cdot\binom{s}{r'-i}_q\right|\\
&\leq D(q)^3C(|q|)^3[r'+s']_q!\sum_{i=0}^{r'}|q|^{(r-i)(r'-i)+k(r-i)+k(r'-i)}\\
&\leq D(q)^3C(|q|)^3[r'+s']_q!\cdot |q|^{k(r-r')}\cdot\dfrac{1}{1-|q|}.
\end{align*}
Meanwhile, we have
\begin{align*}
&\left|\sum_{i=r'}^{r'}q^{(r-i)(r'-i)+k(r-i)+k(r'-i)}\cdot[r']_q!\cdot[s']_q!\cdot\binom{r}{i}_q\cdot\binom{s}{r'-i}_q\right|\\
&=|q|^{k(r-r')}[r']_q!\cdot[s']_q!\cdot\binom{r}{r'}_q\\
&=|q|^{k(r-r')}[r+s]_q! \cdot \dfrac{(1-q^{r-r'+1})\cdots (1-q^{r})}{(1-q^{s'+1})\cdots (1-q^{r+s})}\\
&\geq |q|^{k(r-r')}[r+s]_q! \cdot \dfrac{(1-|q|^{r-r'+1})\cdots (1-|q|^{r})}{(1+|q|^{s'+1})\cdots (1+|q|^{r+s})}\\
&\geq \dfrac{|q|^{k(r-r')}}{D(q)C(|q|)}\cdot[r+s]_q!,
\end{align*}
and 
\begin{align*}
&\left|\sum_{i=0}^{r'-1}q^{(r-i)(r'-i)+k(r-i)+k(r'-i)}\cdot[r']_q!\cdot[s']_q!\cdot\binom{r}{i}_q\cdot\binom{s}{r'-i}_q\right|\\
&\leq \sum_{i=0}^{r'-1}|q|^{(r-i)(r'-i)+k(r-i)+k(r'-i)}\cdot[r']_q!\cdot[s']_q!\cdot\binom{r}{i}_q\cdot\binom{s}{r'-i}_q\\
&\leq |q|^{k(r-r')}\sum_{i=0}^{r'-1}|q|^{(r-i)}[r+s]_q!D(q)^3C(|q|)^3\\
&\leq |q|^{k(r-r')}\dfrac{|q|}{1-|q|}D(q)^3C(|q|)^3[r+s]_q!.
\end{align*}
Hence by the triangle inequality,
\begin{align*}
&\left|\sum_{i=0}^{r'}q^{(r-i)(r'-i)+k(r-i)+k(r'-i)}\cdot[r']_q!\cdot[s']_q!\cdot\binom{r}{i}_q\cdot\binom{s}{r'-i}_q\right|\\
&\geq \dfrac{|q|^{k(r-r')}}{D(q)C(|q|)}\cdot[r+s]_q!-\dfrac{|q|^{k(r-r')+1}}{1-|q|}D(q)^3C(|q|)^3[r+s]_q!\\
&=|q|^{k(r-r')}\left(\dfrac{1}{D(q)C(|q|)}-\dfrac{|q|}{1-|q|}D(q)^3C(|q|)^3\right)\cdot [r+s]_q!.
\end{align*}
Finally, if we let $E(q)=\dfrac{1}{D(q)C(|q|)}-\dfrac{|q|}{1-|q|}D(q)^3C(|q|)^3$, $F(q)=\dfrac{D(q)^3C(|q|)^3}{1-|q|}$, the proof is complete.
\end{proof}

The following corollary will be used multiple times later.
\begin{cor}\label{cor: estimates of norms}
Let $r,s,k$ be non-negative integers and let $\xi\in T^k$ be of norm 1.  Then there is a positive number $\alpha>0$, such that whenever $|q|\leq \alpha$, we have
\begin{equation}\label{formula: estimate of norms}
\dfrac{1}{2}[r+s]_q!\leq\|\xi_{r,s}\|^2_q\leq 2[r+s]_q!.
\end{equation}
\end{cor}

\begin{rem}
Lemma \ref{Estimate of inner products} and Corollary \ref{cor: estimates of norms} hold when $-1/7<q<1/4.$
\end{rem}

To prove the main theorem of this section, we need another lemma.
\begin{lem}\label{lem: positivity of matrix}
Let $\alpha\in \mathbb{R}$ with $|\alpha|<1$. For any $n\in \mathbb{N}$ we define
\[
E_{\alpha}=
  \begin{pmatrix}
    0 & -\alpha & \cdots & -\alpha^{n-1} \\
    -\alpha & \ddots & \ddots & \vdots \\
    \vdots & \ddots & \ddots & -\alpha\\
    -\alpha^{n-1} & \cdots & -\alpha & 0 
  \end{pmatrix},
\]
then the operator norm $\|E_{\alpha}\|\leq \dfrac{2|\alpha|}{1-|\alpha|}$.
\end{lem}
\begin{proof}
Clearly we have 
\[
  \left\|\begin{pmatrix}
   0 & -\alpha^{k} && \mbox{\Huge 0}\\
   & \ddots & \ddots &  \\
      && \ddots& -\alpha^{k}\\
     \mbox{\Huge 0} & &&0

  \end{pmatrix}\right\|=|\alpha|^k.
\]
Hence $\|E_{\alpha}\|\leq 2\sum_{i=1}^{n-1}|\alpha|^i\leq \dfrac{2|\alpha|}{1-|\alpha|}.$
\end{proof}

Take an orthonormal basis $\{\xi^i_j:j\in I_k\}$ for $T_k, k\geq 1$. We may re-order the set $\cup_{i\in \mathbb{N}}\{\xi^i_j:j\in I_k\}$ as $\{\xi^i:i\in I\}$ for some index set $I$ and we set that $\xi^0=\Omega.$

Finally we are ready to state and prove the main result of this section.
\begin{thm}\label{thm: Radulescu basis}
For $-1<q<1$ with $|q|$ sufficiently small, the set $\left\{\dfrac{\xi^i_{r,s}}{\|\xi^i_{r,s}\|}:i\in I, r,s \geq 0\right\}$ forms a Rietz basis for $L^2(M)\ominus L^2(A)$, i.e., $span\left\{\dfrac{\xi^i_{r,s}}{\|\xi^i_{r,s}\|}:i\in I, r,s \geq 0\right\}$ is dense in $L^2(M)\ominus L^2(A)$ and there exists some constants $A_q, B_q>0$, such that for all $\lambda^i_{r,s}\in \mathbb{C}$, one has
\begin{equation}\label{formula: rietz basis}
A_q\sum_{r,s,i}|\lambda^i_{r,s}|^2\|\xi^i_{r,s}\|^2\leq \|\sum_{r,s,i}\lambda^i_{r,s}\xi^i_{r,s}\|^2\leq B_q\sum_{r,s,i}|\lambda^i_{r,s}|^2\|\xi^i_{r,s}\|^2.
\end{equation}
\end{thm}
\begin{proof}
 By Lemma \ref{basic orthogonality}, it suffices to find such $A_q, B_q>0$ which are independent of $i\in I$ and $k\geq 0$ such that 
\[A_q\sum_{r+s=k}|\lambda^i_{r,s}|^2\|\xi^i_{r,s}\|^2\leq \|\sum_{r+s=k}\lambda^i_{r,s}\xi^i_{r,s}\|^2\leq B_q\sum_{r+s=k}|\lambda^i_{r,s}|^2\|\xi^i_{r,s}\|^2\]
holds for any $\lambda^{i}_{r,s}\in \mathbb{C}$. Fixing $\xi=\xi^i\in T^t$ for some $t\in \mathbb{N}$, for simplicity, we will omit the superscript $i$ in the rest of the proof. We fix an $\epsilon>0$ small enough such that $|q|\leq \epsilon $  implies that $1/2\leq E(q),F(q)\leq 2$ (in particular, Corollary \ref{cor: estimates of norms} holds). Then for all $q$ with $|q|\leq \epsilon$,
\begin{equation*}
\begin{split}
\|\sum_{r+s=k}\lambda_{r,s}\xi_{r,s}\|^2
&\geq \sum_{r+s=k}|\lambda_{r,s}|^2\|\xi_{r,s}\|^2-\sum_{r+s=r'+s'=k, r\neq r'}|\lambda_{r,s}\lambda_{r',s'}|\cdot|\left\langle \xi_{r,s},\xi_{r',s'}\right\rangle|\\
&\stackrel{(\ref{formula: estimate of inner product})}{\geq}\sum_{r+s=k}|\lambda_{r,s}|^2\|\xi_{r,s}\|^2-2\sum_{r+s=r'+s'=k, r\neq r'}|\lambda_{r,s}\lambda_{r',s'}|\cdot |q|^{t|r-r'|}[k]_q!\\
&\geq \sum_{r+s=k}|\lambda_{r,s}|^2\|\xi_{r,s}\|^2-2\sum_{r+s=r'+s'=k, r\neq r'}|\lambda_{r,s}\lambda_{r',s'}|\cdot |q|^{|r-r'|}[k]_q!\\
&\stackrel{(\ref{formula: estimate of norms})}{\geq} \sum_{r+s=k}|\lambda_{r,s}|^2\|\xi_{r,s}\|^2-4\sum_{r+s=r'+s'=k, r\neq r'}|\lambda_{r,s}\lambda_{r',s'}|\cdot |q|^{|r-r'|}\|\xi_{r,s}\|\|\xi_{r',s'}\|\\
&=\left\langle (1+4E_{|q|})\begin{pmatrix}
|\lambda_{k,0}|\|\xi_{k,0}\|\\
\vdots\\
|\lambda_{0,k}|\|\xi_{0,k}\|
\end{pmatrix},\begin{pmatrix}
|\lambda_{k,0}|\|\xi_{k,0}\|\\
\vdots\\
|\lambda_{0,k}|\|\xi_{0,k}\|
\end{pmatrix}\right\rangle, 
\end{split}
\end{equation*}
where $E_{|q|}$ is the matrix defined in the previous lemma. As $q$ approaches 0, $4\|E_{|q|}\|\rightarrow 0$ by the previous lemma, $1+4E_{|q|}$ will become strictly positive once $q$ is close enough to $0$. Also, notice that the strict positivity of $1+4E_{|q|}$ depends on neither $i$ nor $k$. This shows the existence of $A_q>0$ satisfying the first half of (\ref{formula: rietz basis}).

Similarly,
\begin{align*}
\|\sum_{r+s=k}\lambda_{r,s}\xi_{r,s}\|^2
&\leq \sum_{r+s=k}|\lambda_{r,s}|^2\|\xi_{r,s}\|^2+\sum_{r+s=r'+s'=k, r\neq r'}|\lambda_{r,s}\lambda_{r',s'}|\cdot|\left\langle \xi_{r,s},\xi_{r',s'}\right\rangle|\\
&\stackrel{(\ref{formula: estimate of inner product})}{\leq}\sum_{r+s=k}|\lambda_{r,s}|^2\|\xi_{r,s}\|^2+2\sum_{r+s=r'+s'=k, r\neq r'}|\lambda_{r,s}\lambda_{r',s'}|\cdot |q|^{t|r-r'|}[k]_q!\\
&\leq\sum_{r+s=k}|\lambda_{r,s}|^2\|\xi_{r,s}\|^2+2\sum_{r+s=r'+s'=k, r\neq r'}|\lambda_{r,s}\lambda_{r',s'}|\cdot |q|^{|r-r'|}[k]_q!\\
&\stackrel{(\ref{formula: estimate of norms})}{\leq} \sum_{r+s=k}|\lambda_{r,s}|^2\|\xi_{r,s}\|^2+4\sum_{r+s=r'+s'=k, r\neq r'}|\lambda_{r,s}\lambda_{r',s'}|\cdot |q|^{|r-r'|}\|\xi_{r,s}\|\|\xi_{r',s'}\|\\
&=\left\langle (1-4E_{|q|})\begin{pmatrix}
|\lambda_{k,0}|\|\xi_{k,0}\|\\
\vdots\\
|\lambda_{0,k}|\|\xi_{0,k}\|
\end{pmatrix},\begin{pmatrix}
|\lambda_{k,0}|\|\xi_{k,0}\|\\
\vdots\\
|\lambda_{0,k}|\|\xi_{0,k}\|
\end{pmatrix}\right\rangle, 
\end{align*}
and the existence of $B_q$ is obvious.

The completeness is already shown in Lemma \ref{completeness}, therefore we are done.

\end{proof}
\begin{rem}
By Lemma \ref{lem: positivity of matrix}, $1\pm 4E_q$ is strictly positive when $|q|<1/9$.
\end{rem}
\begin{rem}
Recall that $\xi^0=\Omega$. If we consider $\left\{\dfrac{\xi^0_{r,0}}{\|\xi^0_{r,0}\|}:r\geq 0\right\}\cup\left\{\dfrac{\xi^i_{r,s}}{\|\xi^i_{r,s}\|}:i\in I, r,s \geq 0\right\}$, then this is a Rietz basis of the entire $q$-Fock space $L^2(M)$.
\end{rem}

\section{Locating the supports of elements in the relative commutant}

Throughout this section we will assume that $-1<q<1$ is a real number with $|q|$ small enough such that the conclusions in Corollary \ref{cor: estimates of norms} and Theorem \ref{thm: Radulescu basis} hold.

For $N\geq 0$, define the idempotent $L_N, R_N: L^2(M)\rightarrow L^2(M)$ by $L_N|_{\F_q(\mathbb{R}e)}=R_N|_{\F_q(\mathbb{R}e)}=0$ and
\begin{equation}
\begin{split} 
L_N(\sum_{i\in I,r,s\geq 0} c_{r,s}^i\xi^i_{r,s})=\sum_{i\in I, s\geq 0, 0\leq r\leq N}c_{r,s}^i\xi^i_{r,s},\\
R_N(\sum_{i\in I,r,s\geq 0} c_{r,s}^i\xi^i_{r,s})=\sum_{i\in I, r\geq 0, 0\leq s\leq N}c_{r,s}^i\xi^i_{r,s}.
\end{split}
\end{equation}
By Theorem \ref{thm: Radulescu basis}, $L_N$ and $R_N$ are both well-defined. Moreover, with a little abuse of notation, sometimes we will also use $L_N (\text{resp. }R_N)$ to denote the image of $L_N (\text{resp. }R_N)$.

Let $C\subset A$ be a diffuse subalgebra and fix a free ultrafilter $\omega\in \beta(\mathbb{N})\backslash\mathbb{N}$. Let $z=(z_n)_n\in M^{\omega}\ominus A^{\omega}\cap C'.$ Without loss of generality we assume that $\|z_n\|\leq 1$ (the operator norm is bounded above by 1) and $x_n\in M\ominus A, \forall n$. Just as in \cite{wen15radialmasa}, we would like to show that the support of $z$ eventually escapes both $L_N$ and $R_N$. To this end, we need some preparations.

The first key step towards our goal is to show that $L_N$ is asymptotically right-$A$ modular. 

Recall that $Q_k$ is the orthogonal projection from $\F_q(\H_{\mathbb{R}})$ onto $\H^{\otimes k}$.

\begin{lem}\label{z escapes Q_k}
For any $k\in \mathbb{N}$, we have 
\[\lim_{n\to \omega}Q_k(z_n)=0.\]
\end{lem}
\begin{proof}
Suppose this is not the case, then there exists some $k\in\mathbb{N}$ with
 \[\H^{\otimes k} \ni z_0=\lim_{n\to \omega}Q_k(z_n)\neq 0.\]
 In particular, $z_n\to z$ weakly for some non-zero $z\in M$. Clearly $z\in C'\cap M\ominus A$.
 
However, it is known that the generator masa $A$ is \textit{mixing} in $M$(see \cite{BikramMukherjee16generatormasa},\cite{wen2016singularityinq-gaussian}). Thus by Proposition 5.1 in \cite{CFM13mixing} we must have that $C'\cap M=A$, a contradiction.
\end{proof}

The next estimate will be essential in order to establish the right-$A$ modularity of $L_N$.
\begin{lem}\label{lem:estimate for right-A modularity}
Let $x\in L^2(M)\ominus L^2(A)$ whose Fourier expansion along $\{\xi_{r,s}^i:i\in I, r,s\geq 0\}$ is of the form
\[x=\sum_{r\geq N+1,s\geq 0}\lambda_{r,s}\xi_{r,s},\]
where $\xi=\xi^i$ for some fixed $i\in I$ with $\xi \in T^t$. Then we have 
\begin{equation}
\begin{split}
\left\|L_N\left(a_r(e)^kx\right)\right\|^2_2\leq \dfrac{4kB_qC(|q|)^3D(q)^6}{(1-q)^{k}(1-q^{2t})} \sum_{r,s\geq 0}q^{2(t+s+r-k-N-1)}|\lambda_{r,s}|^2\|\xi_{r,s}\|_2^2,
\end{split}
\end{equation}
for all $k,N\geq 0.$
\end{lem}
\begin{proof}
 We let $\lambda_{r,s}=0$ for all $r\leq N$ and $s\geq 0$. By (\ref{annilators acting on basis}), we have
 \begin{equation*}
 \begin{split}
 & \quad L_N\left(a_r(e)^k\sum_{r\geq N+1,s\geq 0}\lambda_{r,s}\xi_{r,s}\right)\\
 &=L_N\left(\sum_{r\geq N+1,s\geq 0}\lambda_{r,s}\sum_{i+j=k,i,j\geq 0}\dfrac{[r]_q!}{[r-i]_q!}\cdot \dfrac{[s]_q!}{[s-j]_q!}\cdot\binom{k}{i}_qq^{(t+s-j)i}\xi_{r-i,s-j}\right)\\
 &=\sum_{r\leq N, s\geq 0}\xi_{r,s}\sum_{i+j=k,r+i\geq N+1,j\geq 0}\lambda_{r+i,s+j}\dfrac{[r+i]_q!}{[r]_q!}\cdot \dfrac{[s+j]_q!}{[s]_q!}\cdot\binom{k}{i}_qq^{(t+s)i}.
 \end{split}
 \end{equation*}
 Note that for all $-1<q<1$,
 \begin{itemize}
 \item $\dfrac{[r+i]_q!}{[r]_q!}\cdot \dfrac{[s+j]_q!}{[s]_q!}\leq \dfrac{D(q)^2}{(1-q)^k},\quad \forall i+j=k$;
 \item $\binom{k}{i}_q\leq D(q)C(|q|)$.
 \end{itemize}
 Therefore, for each $r\leq N$,
 \begin{align*}
&\quad\left|\sum_{i+j=k,r+i\geq N+1,j\geq 0}\lambda_{r+i,s+j}\dfrac{[r+i]_q!}{[r]_q!}\cdot \dfrac{[s+j]_q!}{[s]_q!}\cdot\binom{k}{i}_qq^{(t+s)i}\right|^2\\
 &\leq \left|\sum_{i+j=k,r+i\geq N+1,j\geq 0}|\lambda_{r+i,s+j}|\|\xi_{r+i,s+j}\|_2\cdot \dfrac{1}{\|\xi_{r+i,s+j}\|_2}\dfrac{C(|q|)D(q)^3}{(1-q)^k}|q|^{(t+s)i}\right|^2\\
 &\leq \left(\sum_{i+j=k,r+i\geq N+1,j\geq 0}|\lambda_{r+i,s+j}|^2\|\xi_{r+i,s+j}\|_2^2\right)\left(\sum_{i+j=k,r+i\geq N+1,j\geq 0}\dfrac{1}{\|\xi_{r+i,s+j}\|_2^2}\dfrac{C(|q|)^2D(q)^6}{(1-q)^{2k}}q^{2(t+s)i}\right)\\
 &\stackrel{(\ref{formula: estimate of norms})}{\leq} \left(\sum_{i+j=k,r+i\geq N+1,j\geq 0}|\lambda_{r+i,s+j}|^2\|\xi_{r+i,s+j}\|_2^2\right) \dfrac{2}{[r+s+k]_q!}\cdot\dfrac{C(|q|)^2D(q)^6}{(1-q)^{2k}}\sum_{i+j=k,r+i\geq N+1,j\geq 0}q^{2(t+s)i}\\
 &\leq \left(\sum_{i+j=k,r+i\geq N+1,j\geq 0}|\lambda_{r+i,s+j}|^2\|\xi_{r+i,s+j}\|_2^2\right) \dfrac{2}{[r+s+k]_q!}\cdot\dfrac{C(|q|)^2D(q)^6}{(1-q)^{2k}}\cdot\dfrac{q^{2(t+s)(N+1-r)}}{1-q^{2(t+s)}}\\
 &\leq \left(\sum_{i+j=k,r+i\geq N+1,j\geq 0}|\lambda_{r+i,s+j}|^2\|\xi_{r+i,s+j}\|_2^2\right)\dfrac{2C(|q|)^2D(q)^6}{(1-q)^{2k}(1-q^2)}\cdot\dfrac{q^{2(t+s+r-N-1)}}{[r+s+k]_q!},
  \end{align*}
  where in the last inequality we used the fact that $ab\geq a-b$ for all $a,b\geq 1$.
  
 Finally, we have 
 \begin{align*}
& \quad \left\|L_N\left(a_r(e)^kx\right)\right\|^2_2\\
&=\left\|\sum_{r\leq N, s\geq 0}\xi_{r,s}\sum_{i+j=k,r+i\geq N+1,j\geq 0}\lambda_{r+i,s+j}\dfrac{[r+i]_q!}{[r]_q!}\cdot \dfrac{[s+j]_q!}{[s]_q!}\cdot\binom{k}{i}_qq^{(t+s)i}\right\|_2^2\\
&\leq B_q\sum_{r\leq N, s\geq 0}\|\xi_{r,s}\|_2^2\left|\sum_{i+j=k,r+i\geq N+1,j\geq 0}\lambda_{r+i,s+j}\dfrac{[r+i]_q!}{[r]_q!}\cdot \dfrac{[s+j]_q!}{[s]_q!}\cdot\binom{k}{i}_qq^{(t+s)i}\right|^2\\
&\leq B_q\sum_{r\leq N, s\geq 0}\|\xi_{r,s}\|_2^2\cdot\left(\sum_{i+j=k,r+i\geq N+1,j\geq 0}|\lambda_{r+i,s+j}|^2\|\xi_{r+i,s+j}\|_2^2\right)\dfrac{2C(|q|)^2D(q)^6}{(1-q)^{2k}(1-q^2)}\cdot\dfrac{q^{2(t+s+r-N-1)}}{[r+s+k]_q!}\\
&\stackrel{(\ref{formula: estimate of norms})}{\leq} 2B_q\sum_{r\leq N, s\geq 0}[r+s]_q!\left(\sum_{i+j=k,r+i\geq N+1,j\geq 0}|\lambda_{r+i,s+j}|^2\|\xi_{r+i,s+j}\|_2^2\right)\dfrac{2C(|q|)^2D(q)^6}{(1-q)^{2k}(1-q^2)}\cdot\dfrac{q^{2(t+s+r-N-1)}}{[r+s+k]_q!}\\
&\leq 4B_q\sum_{r\leq N, s\geq 0}\left(\sum_{i+j=k,r+i\geq N+1,j\geq 0}|\lambda_{r+i,s+j}|^2\|\xi_{r+i,s+j}\|_2^2\right)\dfrac{C(|q|)^3D(q)^6}{(1-q)^{k}(1-q^2)}\cdot q^{2(t+s+r-N-1)}\\
&\leq \dfrac{4kB_qC(|q|)^3D(q)^6}{(1-q)^{k}(1-q^2)} \sum_{r,s\geq 0}q^{2(t+s+r-k-N-1)}|\lambda_{r,s}|^2\|\xi_{r,s}\|_2^2,
 \end{align*}
 where the second last inequality is due to the fact that 
 \[\dfrac{[r+s]_q!}{[r+s+k]_q!}\leq (1-q)^kC(|q|).\]
\end{proof}

\begin{prop}\label{right-A modularity}
For all $N,m\in \mathbb{N}$ and for any $x=(x_n)\in (L^2(M\ominus A))^{\omega}$ such that $\lim _{n\to \omega}Q_k(x_n)\to 0, \forall k\in \mathbb{N}$, we have 
\[\lim_{n\to \omega}\|L_N(s_r(e^{\otimes m})x_n)-s_r(e^{\otimes m})L_N(x_n)\|_2=0.\] 
In particular, for all unitary $u$ in the C$^*$-algebra $C^*(s(e))$ generated by $s(e)$, $N\in \mathbb{N}$ and $(z_n)\in M^{\omega}\ominus A^{\omega}\cap C'$, 
\[\lim_{n\to \omega}\|L_N(z_nu)-L_N(z_n)u\|_2=0.\]
\end{prop}

\begin{proof}
Each $s_r(e^{\otimes m})$ can be written as a finite linear combination of $c_r(e)^ka_r(e)^l$ with $k,l\geq 0$ and $k+l=m$, hence it suffices to show 
\[\lim_{n\to \omega}\left\|L_N(c_r(e)^ka_r(e)^lx_n)-c_r(e)^ka_r(e)^lL_N(x_n)\right\|_2=0,\]
for all $k,l\geq 0, k+l=m$. 

Since $c_r(e)L_N(\xi)=L_N(c_r(e)\xi)$ for all $\xi \in L^2(M\ominus A)$, it then reduces to prove
\[\lim_{n\to \omega}\left\|L_N(a_r(e)^kx_n)-a_r(e)^kL_N(x_n)\right\|_2=0,\]
for all $0\leq k\leq m$.

Suppose that $x_n=\sum_{i\in I, r,s\geq 0}\lambda^{i,n}_{r,s}\xi_{r,s}^i$ is the Fourier decomposition along the Riesz basis $\{\xi^i_{r,s}\}$, observe that 
\begin{align*}
L_N(a_r(e)^kx_n)-a_r(e)^kL_N(x_n)&=L_N(a_r(e)^kx_n)-L_N\left(a_r(e)^kL_N(x_n)\right)\\
&=L_N\left(a_r(e)^k(1-L_N)(x_n)\right)\\
&=L_N\left(a_r(e)^k\sum_{i\in I, r\geq N+1,s\geq 0}\lambda^{i,n}_{r,s}\xi_{r,s}^i\right)\\
&=\sum_{i\in I}L_N\left(a_r(e)^k\sum_{ r\geq N+1,s\geq 0}\lambda^{i,n}_{r,s}\xi_{r,s}^i\right).
\end{align*}
By Lemma \ref{lem:estimate for right-A modularity}, we have 
\begin{align*}
\|L_N(a_r(e)^kx_n)-a_r(e)^kL_N(x_n)\|^2_2
&=\sum_{i\in I}\left\|L_N\left(a_r(e)^k\sum_{ r\geq N+1,s\geq 0}\lambda^{i,n}_{r,s}\xi_{r,s}^i\right)\right\|_2^2\\
&\leq \sum_{i\in I}\dfrac{4kB_qC(|q|)^3D(q)^6}{(1-q)^{k}(1-q^{2})} \sum_{r,s\geq 0}q^{2(|\xi^i|+s+r-k-N-1)}|\lambda_{r,s}^{i,n}|^2\|\xi_{r,s}^i\|_2^2.
\end{align*}
We may assume that for each $n$, there exists a natural number $t_n$ such that: (1) $t_n\to \infty$ as $n\to\omega$ and (2) $\lambda_{r,s}^{i,n}=0$ for any $i,r,s$ with $|\xi_i|+r+s\leq t_n+N+1$. Thus

\begin{align*}
\|L_N(a_r(e)^kx_n)-a_r(e)^kL_N(x_n)\|^2_2
&\leq\dfrac{4kB_qC(|q|)^3D(q)^6}{(1-q)^{k}(1-q^{2})}\cdot q^{2(t_n-k)}\sum_{i,r,s }|\lambda_{r,s}^{i,n}|^2\|\xi_{r,s}^i\|_2^2\\
&\leq \dfrac{4kB_qC(|q|)^3D(q)^6}{(1-q)^{k}(1-q^{2})A_q}\cdot q^{2(t_n-k)}\|x_n\|_2^2.
\end{align*}
As $n\to \omega$, $t_n$ diverges to infinity. Therefore the right-hand side of the above inequality converges to 0 (uniformly on the unit ball of $L^2(M^{\omega}\ominus A^{\omega})$).

The case for  $u\in C^*(s(e))$ and $z=(z_n)\in M^{\omega}\ominus A^{\omega}\cap C'$ then is an easy consequence of Lemma \ref{z escapes Q_k} and the fact that $\{s(e^{\otimes n}):n\geq 0\}$ spans norm-densely in $C^*(s(e))$.
\end{proof}

Our next step is to show that for any $x\in M\ominus A$ and for all sequence of unitary elements $(u_k)_k$ in $C^*(s(e))$ which goes to 0 weakly, $(u_kL_N(x))_k$ is asymptotically orthogonal to the subspace $L_N$.

\begin{prop}\label{prop:s(e^n) maps z orthogonal to L_N}
There exists a positive number $G(q)>0$  such that 
\begin{align}
\left \|L_{N_1}\left(s(e^{\otimes n})L_{N_2}(x)\right)\right\|_2\leq G(q)\cdot (n+1)^{3/2}\cdot |q|^{(n-N_1-N_2)}\|x\|_2
\end{align}
for any $N_1,N_2\in \mathbb{N}$ and $x\in M\ominus A$. The choice of $G(q)$ is independent of $N_1,N_2,x$.
\end{prop}  
\begin{proof}
By the Wick formula (\ref{Wick formula}), one has 
\begin{align*}
L_{N_1}\left (s(e^{\otimes n})L_{N_2}(x)\right)=\sum_{k=0}^{n}\binom{n}{k}_qL_{N_1}\left(c(e)^ka(e)^{n-k}L_{N_2}(x)\right).
\end{align*}
Also, notice that in the above summation, only the terms with $0\leq k\leq N_1$ will be able to contribute something non-zero. 

Let us estimate each summand for $0\leq k\leq N_1$. Suppose $x=\sum_{i}\sum_{r,s}\lambda^i_{r,s}\xi^i_{r,s}$ be the Fourier decomposition along $\{\xi^i_{r,s}\}$, then  
\begin{align*}
&\quad L_{N_1}\left(c(e)^ka(e)^{n-k}L_{N_2}(x)\right)\\
&=\sum_iL_{N_1}\left(c(e)^k\sum_{r\leq N_2,s\geq 0}\lambda^i_{r,s}\sum_{j_1+j_2=n-k,j_2\leq r}\dfrac{[r]_q!}{[r-j_2]_q!} \dfrac{[s]_q!}{[s-j_1]_q!}\cdot\binom{n-k}{j_1}_q\cdot q^{(|\xi^i|+r-j_2)j_1}\xi^i_{r-j_2,s-j_1}\right)\\
&=\sum_iL_{N_1}\left(\sum_{r\leq N_2,s\geq 0}\lambda^i_{r,s}\sum_{j_1+j_2=n-k,j_2\leq r}\dfrac{[r]_q!}{[r-j_2]_q!} \dfrac{[s]_q!}{[s-j_1]_q!}\cdot\binom{n-k}{j_1}_q\cdot q^{(|\xi^i|+r-j_2)j_1}\xi^i_{r+k-j_2,s-j_1}\right)\\
&=\sum_i\left(\sum_{k\leq r\leq N_1, s\geq 0}\xi^i_{r,s}\sum_{\substack{j_1+j_2=n-k,\\j_1\geq n-2k+r-N_2}}\lambda^i_{r+j_2-k,s+j_1}\dfrac{[r+j_2-k]_q!}{[r-k]_q!} \dfrac{[s+j_1]_q!}{[s]_q!}\cdot\binom{n-k}{j_1}_q\cdot q^{(|\xi^i|+r-k)j_1}\right).
\end{align*}
Now for $k\leq r\leq N_1$,
\begin{align*}
&\quad\left|\sum_{\substack{j_1+j_2=n-k, \\j_1\geq n-2k+r-N_2}}\lambda^i_{r+j_2-k,s+j_1}\dfrac{[r+j_2-k]_q!}{[r-k]_q!} \dfrac{[s+j_1]_q!}{[s]_q!}\cdot\binom{n-k}{j_1}_qq^{(|\xi^i|+r-k)j_1}\right|^2\\
&\leq \sum_{\substack{j_1+j_2=n-k, \\j_1\geq n-2k+r-N_2}}|\lambda^i_{r+j_2-k,s+j_1}|\cdot\dfrac{D(q)^2}{(1-q)^{n-k}}\cdot C(|q|)D(q)\cdot|q|^{j_1}\\
&\leq \left(\sum_{\substack{j_1+j_2=n-k,\\j_1\geq n-2k+r-N_2}}|\lambda_{r+j_2-k,s+j_1}^i|^2\|\xi_{r+j_2-k,s+j_1}\|_2^2\right)\left(\sum_{\substack{j_1+j_2=n-k,\\j_1\geq n-2k+r-N_2}}\dfrac{1}{\|\xi_{r+j_2-k,s+j_1}\|_2^2}\dfrac{C(|q|)^2D(q)^6}{(1-q)^{2(n-k)}}q^{2j_1}\right)\\
 &\leq \left(\sum_{\substack{j_1+j_2=n-k, \\j_1\geq n-2k+r-N_2}}|\lambda_{r+j_2-k,s+j_1}^i|^2\|\xi_{r+j_2-k,s+j_1}\|_2^2\right) \dfrac{2}{[r+s+n-2k]_q!}\dfrac{C(|q|)^2D(q)^6}{(1-q)^{2(n-k)}}\cdot\dfrac{q^{2(n-N_1-N_2)}}{1-q^2}\\
 &\leq \left(\sum_{\substack{j_1+j_2=n-k, \\j_1\geq n-2k+r-N_2}}|\lambda_{r+j_2-k,s+j_1}^i|^2\|\xi_{r+j_2-k,s+j_1}\|_2^2\right)\dfrac{2C(|q|)^2D(q)^6}{(1-q)^{2(n-k)}(1-q^2)[r+s+n-2k]_q!}\cdot q^{2(n-N_1-N_2)}.
\end{align*}
Hence 
\begin{align*}
&\quad \left\|L_{N_1}\left(c(e)^ka(e)^{n-k}L_{N_2}(x)\right)\right\|_2^2\\
&=\sum_i \left\| \sum_{k\leq r\leq N_1, s\geq 0}\xi^i_{r,s}\sum_{\substack{j_1+j_2=n-k, \\j_1\geq n-2k+r-N_2}}\lambda^i_{r+j_2-k,s+j_1}\dfrac{[r+j_2-k]_q!}{[r-k]_q!} \dfrac{[s+j_1]_q!}{[s]_q!}\cdot\binom{n-k}{j_1}_q\cdot q^{(|\xi^i|+r-k)j_1}\right\|_2^2\\
&\leq B_q\sum_i \sum_{k\leq r\leq N_1, s\geq 0}\|\xi^i_{r,s}\|_2^2\cdot\left|\sum_{\substack{j_1+j_2=n-k, \\j_1\geq n-2k+r-N_2}}\lambda^i_{r+j_2-k,s+j_1}\dfrac{[r+j_2-k]_q!}{[r-k]_q!} \dfrac{[s+j_1]_q!}{[s]_q!}\cdot\binom{n-k}{j_1}_q\cdot q^{(|\xi^i|+r-k)j_1}\right|^2\\
&\leq B_q\sum_i \sum_{k\leq r\leq N_1, s\geq 0}\left(\sum_{\substack{j_1+j_2=n-k, \\j_1\geq n-2k+r-N_2}}|\lambda_{r+j_2-k,s+j_1}^i|^2\|\xi_{r+j_2-k,s+j_1}\|_2^2\right)\cdot\dfrac{4C(|q|)^3D(q)^6}{(1-q)^{n}(1-q^2)}\cdot q^{2(n-N_1-N_2)}\\
&\leq \dfrac{4(n+1)B_qC(|q|)^3D(q)^6}{(1-q)^{n}(1-q^2)A_q}\cdot q^{2(n-N_1-N_2)} \|x\|_2^2.
\end{align*}

Finally, using the rough estimate $\binom{n}{k}_q\leq C(|q|)D(q)$ and the triangle inequality, we conclude that 
\[\left \|L_{N_1}\left(s(e^{\otimes n})L_{N_2}(x)\right)\right\|_2\leq (n+1)C(|q|)D(q)\sqrt{\dfrac{4(n+1)B_qC(|q|)^3D(q)^6}{(1-q)^{n}(1-q^2)A_q}}\cdot |q|^{(n-N_1-N_2)} \|x\|_2.\]
Set $G(q)=\dfrac{2B_q^{1/2}C(|q|)^{5/2}D(q)^4}{(1-q)^{n/2}(1-q^2)^{1/2}A_q^{1/2}}$, we are done.
\end{proof}


Finally, we are ready to prove our main result in the section.

\begin{thm}\label{thm: essential support}
For all $N\in \mathbb{N}$ and $z=(z_n)_n\in M^{\omega}\ominus A^{\omega}\cap C'$, we have 
\begin{align}
\lim_{n\to \omega}\|L_N(z_n)\|_2=\lim_{n\to \omega}\|R_N(z_n)\|_2=0.
\end{align}
\end{thm}
\begin{proof}
The proof is similar to the ones in \cite{popa83maxinjective}\cite{wen15radialmasa}, but for completeness we include a sketch.

 Fix $N\in \mathbb{N}$. Let $(u_n)_n$ be a sequence of unitary elements in $C$ which converges to 0 weakly. Let $(u'_n)_n$ be a sequence of unitaries in $C^*(s(e))$ such that $\|u_n-u'_n\|_2\leq \dfrac{1}{2^n}$. Then as in the proof of 
\cite[Lemma 9]{wen15radialmasa}, we further approximate $u'_n$ with finitely supported elements in $C^*(s(e))$:\\
 \\
 \textbf{Claim:} There exists a sequence $(v_k)$ of finitely supported  elements in $C^*(s(e))$ and increasing sequences of natural numbers $(n_k)$ and $(M_k)$, a positive constant $H(q)$,  such that 
 \[\|u'_{n_k}-v_k\|_2\leq \dfrac{1}{2^k}, \quad \left\|L_N\left(v_k\left(L_{M_{k+1}}-L_{M_k}\right)(x)\right)-L_N(v_kx)\right\|_2<H(q)|q|^{2k/3}\|x\|_2\] 
 for all $x\in M\ominus A$.\\
\textit{Proof of the Claim.} We choose $(v_k)$, $(n_k)$ and $(M_k)$ inductively. Assume that we have already chosen them up to  $\{M_k,n_{k-1},v_{k-1}\}$. Since $u_n\to 0$ weakly, if we choose a large $n_k$, then it is possible to well-approximate $u'_{n_k}$ in operator-norm with a finitely supported element $v_k$ in $C^*(s(e))$ whose support is contained in $[M_k+N+k+1,N_k]$ for some $N_k\in \mathbb{N}$. By Proposition \ref{prop:s(e^n) maps z orthogonal to L_N}, we have $\|L_N(v_kL_{M_k}(x))\|_2\leq G(q)|q|^{2(k+1)/3}\|x\|_2$ for all $x\in M\ominus A$, where $G(q)$ is a positive constant which only depends on $q$. If we take $M_{k+1}\geq N_k+N+1$ large enough, then we have $\|L_N(v_k(L_{M_{k+1}}(x)-x))\|_2<|q|^k$. Thus we are done. 
 \\
 \\
 Let us continue with the proof. On one hand, using $z\in C'\cap M^{\omega}$ and the asymptotic right-$A$ modularity of $L_N$ as in Proposition \ref{right-A modularity}, we have 
 \begin{align*}
\lim_{n\to\omega}\sum_{k=N_1}^{N_2}\left\langle L_{N}(u_{n_k} z_n),L_{N}(u_{n_k}z_n)\right\rangle
&=\lim_{n\to\omega}\sum_{k=N_1}^{N_2}\left\langle L_{N}( z_nu_{n_k}),L_{N}(z_nu_{n_k})\right\rangle  \\
&\geq \lim_{n\to\omega}\sum_{k=N_1}^{N_2} \left(\left\langle L_{N}( z_nu'_{n_k}),L_{N}(z_nu'_{n_k})\right\rangle-\left\|L_N(z_n(u_{n_k}-u'_{n_k}))\right\|_2^2\right)\\
&\geq\lim_{n\to\omega}\sum_{k=N_1}^{N_2}\left(\left\langle L_{N}( z_n)u'_{n_k},L_{N}(z_n)u'_{n_k}\right\rangle-\dfrac{\|L_N\|^2\|z_n\|^2}{2^{2n_k}}\right)\\
&\geq(N_2-N_1)\lim_{n\to \omega}\left(\left\|L_{N}(z_n)\right\|_2^2-\dfrac{\|L_N\|^2\|z_n\|^2}{2^{2N_1}}\right).
\end{align*}
On the other hand,
\begin{align*}
\lim_{n\to\omega}\sum_{k=N_1}^{N_2}\left\langle L_{N}\left(u_k z_n\right),L_{N}(u_kz_n)\right\rangle 
&\approx\lim_{n\to\omega}\sum_{k=N_1}^{N_2}\left\langle L_{N}\left(v_k \left(L_{M_{k+1}}-L_{M_k}\right)(z_n)\right),L_{N}\left(v_k\left(L_{M_{k+1}}-L_{M_k}\right)(z_n)\right)\right\rangle\\
&\leq\lim_{n\to\omega} \norm{L_{N}}^2\|v_k\|^2\sum_{k=N_1}^{N_2}\left\langle \left(L_{M_{k+1}}-L_{M_k}\right)(z_n), \left(L_{M_{k+1}}-L_{M_k}\right)(z_n)\right\rangle\\
&\leq \dfrac{4B_q}{A_q}\lim_{n\to\omega} \norm{L_{N}}^2\|z_n\|^2_2.
\end{align*}
By combining the above two estimates and by increasing $N_2-N_1$ and $N_1$, we get the conclusion for $L_N$. The statement about $R_N$ follows by symmetry.
\end{proof}

\section{Strong asymptotic orthogonality property}

\begin{defn}\label{def:S-AOP}
Let $B\subset N$ be an inclusion of finite von Neumann algebras, we say that the inclusion has the \textit{strong asymptotic orthogonality property} (s-AOP for short), if for all $a,b\in N\ominus B$ and $x=(x_n)\in N^{\omega}\ominus B^{\omega}\cap C'$, where $C\subset B$ is any diffuse subalgebra of $B$, then
\[ax\perp xb.\]
\end{defn}

Fix an orthonormal basis $\{e_j:j\in J\}$ of $\H_{\mathbb{R}}$ with $e=e_{j_0}$ for some $j_0\in J$.

Recall that for any $s,t\geq 0$ and for any $\xi\in L^2(M)$, we set 
\[\xi_{s,t}=c(e)^sc_r(e)^t\xi.\]
The following is a direct consequence of the definition of annihilation operators: 
\begin{lem}
For all $x\in L^2(M), N\in\mathbb{N}$ and $j\in J\backslash\{j_0\}$, 
\[a(e_{j})x_{N,N}=q^N\left(a(e_{j})x\right)_{N,N}.\]
\end{lem}
The next estimate is the key technical result of this section.
\begin{lem}\label{a(e_j)x_{N,N} is small}
For all $x\in L^2(M\ominus A), N\in\mathbb{N}$ and $j\in J\backslash\{j_0\}$, we have 
\begin{equation}
\|a(e_{j})x_{N,N}\|_2^2\leq \dfrac{16B_q^2}{A_q^2}\cdot D(q)C(|q|)\cdot q^{2N}\|a(e_{j})\|_{\infty}\|x_{N,N}\|^2_2.
\end{equation}
\end{lem}

\begin{proof}
Suppose that along the Riesz basis $\{\xi_{r,s}^i:i\in I,r,s\geq 0\}\cup\{\xi^0_{r,0}:r\geq 0\}$, we have the Fourier expansions
\begin{align*}
x=\sum_{i,r,s}\lambda^i_{r,s}\xi^{i}_{r,s},\\
a(e_j)x=\sum_{i,r,s}\mu^i_{r,s}\xi^{i}_{r,s}+\sum_{r\geq 0}\mu^0_r\xi^0_{r,0}.
\end{align*}
First note that 
\begin{itemize}
\item $[r+s+2N]_q!\leq \dfrac{D(q)}{(1-q)^{2N}}\cdot [r+s]_q!$ and 
\item $[r+s+2N]_q!\geq \dfrac{1}{(1-q)^{2N}C(|q|)}\cdot [r+s]_q!$.
\end{itemize}

By the previous lemma,
\begin{align*}
\|a(e_{j})x_{N,N}\|_2^2
&=q^{2N}\|\left(a(e_{j})x\right)_{N,N}\|_2^2\\
&=q^{2N}\|\sum_{i,r,s}\mu^i_{r,s}\xi^{i}_{r+N,s+N}+\sum_{r\geq 0}\mu^0_r\xi^0_{r+2N,0}\|_2\\
&\leq q^{2N} B_q\left(\sum_{i,r,s}|\mu^i_{r,s}|^2\|\xi^{i}_{r+N,s+N}\|^2_2+\sum_{r\geq 0}|\mu^0_r|^2\|\xi^0_{r+2N,0}\|^2_2\right)\\
&\leq 2q^{2N} B_q\left(\sum_{i,r,s}|\mu^i_{r,s}|^2[r+s+2N]_q!+\sum_{r\geq 0}|\mu^0_r|^2[r+2N]_q!\right)\\
&\leq \dfrac{2q^{2N}B_qD(q)}{(1-q)^{2N}}\left(\sum_{i,r,s}|\mu^i_{r,s}|^2[r+s]_q!+\sum_{r\geq 0}|\mu^0_r|^2[r]_q!\right)\\
&\leq \dfrac{4q^{2N}B_qD(q)}{(1-q)^{2N}}\left(\sum_{i,r,s}|\mu^i_{r,s}|^2\|\xi^{i}_{r,s}\|_2^2+\sum_{r\geq 0}|\mu^0_r|^2\|\xi^0_{r,0}\|^2_2\right)\\
&\leq \dfrac{4q^{2N}B_qD(q)}{(1-q)^{2N}A_q}\|a(e_{j})x\|_2^2\\
&\leq  \dfrac{4q^{2N}B_qD(q)}{(1-q)^{2N}A_q}\|a(e_{j})\|_{\infty}\|x\|_2^2.
\end{align*}

On the other hand,
\begin{align*}
\|x_{N,N}\|^2_2
&=\|\sum_{i,r,s}\lambda^i_{r,s}\xi^{i}_{r+N,s+N}\|^2_2\\
&\geq A_q\sum_{i,r,s}|\lambda^i_{r,s}|^2\|\xi^{i}_{r+N,s+N}\|^2_2\\
&\geq \dfrac{A_q}{2}\sum_{i,r,s}|\lambda^i_{r,s}|^2[r+s+2N]_q!\\
&\geq \dfrac{A_q}{2(1-q)^{2N}C(|q|)}\sum_{i,r,s}|\lambda^i_{r,s}|^2[r+s]_q!\\
&\geq \dfrac{A_q}{4(1-q)^{2N}C(|q|)}\sum_{i,r,s}|\lambda^i_{r,s}|^2\|\xi^i_{r,s}\|^2_2\\
&\geq \dfrac{A_q}{4(1-q)^{2N}B_qC(|q|)}\|x\|^2_2.
\end{align*}
Combine these two inequalities, we  have 
\begin{align*}
\|a(e_{j})x_{N,N}\|_2^2\leq \dfrac{16B_q^2}{A_q^2}\cdot D(q)C(|q|)\cdot q^{2N}\|a(e_{j})\|_{\infty}\|x_{N,N}\|^2_2.
\end{align*}
\end{proof}

\begin{thm}\label{s-AOP}
The inclusion $A\subset M$ has the strong asymptotic orthogonality property, whenever $|q|$ is sufficiently small. 
\end{thm}
\begin{proof}
Suppose $C\subset A$ is a diffuse subalgebra and $z=(z_n)\in M^{\omega}\ominus A^{\omega}\cap C'$ with $\|z_n\|\leq 1$. By Theorem \ref{thm: essential support}, we can assume that for any $N\in \mathbb{N}$, 
\[\lim_{n\to \omega}\|L_N(z_n)\|_2=\lim_{n\to \omega}\|R_N(z_n)\|_2=0.\]

A density argument reduces the problem to showing that 
\[\lim_{n\to \omega}\left\langle s(e_{i(1)}\otimes \cdots e_{i(k_1)})z_n, s_r(e_{j(1)}\otimes \cdots e_{j(k_2)})z_n \right\rangle=0, \]
for all $k_1,k_2\geq 1$ and $i(1),\cdots,i(k_1), j(1),\cdots,j(k_2)\in J$ such that $\{i(l):1\leq l\leq k_1\}\backslash\{j_0\}\neq \emptyset$ and $\{j(l):1\leq l\leq k_2\}\backslash\{j_0\}\neq \emptyset$.

By the Wick formula (\ref{Wick formula}), it suffices to show the inner product
\begin{align*}
\left\langle  c(e_{i(1)})\cdots c(e_{i(t)})a(e_{i(t+1)})\cdots a(e_{i(k_1)})z_n,
 c_r(e_{j(1)})\cdots c_r(e_{j(s)})a_r(e_{j(s+1)})\cdots a_r(e_{j(k_2)})z_n \right\rangle 
\end{align*}
goes to 0 as $n\to \omega$ for any $0\leq t\leq k_1$ and $0\leq s\leq k_2$. There are two cases:

\textbf{Case 1}:  there exists some $l_1\geq t+1$ with $i(l_1)\neq j_0$, then the previous lemma implies that for any $N\in \mathbb{N}$, one has 
\[\|c(e_{i(1)})\cdots c(e_{i(t)})a(e_{i(t+1)})\cdots a(e_{i(k_1)})z_n\|_2\leq q^{2N}\|z_n\|_2,\]
once $n$ gets sufficiently large (note that here we suppressed the constant before $q^{2N}$ by increasing $n$ if necessary). By letting $N\to \infty$, we clearly have that the inner product goes to 0 as $n\to \omega$.

\textbf{Case 2}:  $i(t+1)=\cdots=i(k_1)=j(s+1)=\cdots=j(k_2)=j_0$. Let $1\leq l_1\leq t$ be the smallest number such that $i(l_1)\neq j_0$. Taking adjoint, we consider 
\[a(e_{i(l_1)})\cdots a(e_{i(1)})c_r(e_{j(1)})\cdots c_r(e_{j(s)})a_r(e_{j(s+1)})\cdots a_r(e_{j(k_2)})z_n \quad\quad (*).\]
A direct computation shows that 
\[a(f_1)c_r(f_2)=c_r(f_2)a(f_1)+\left\langle f_1,f_2 \right\rangle W, \]
where $W\in B(L^2(M))$ is defined by 
\[W|_{\H^{\otimes n}}=q^n Id,\quad \forall n\geq 0.\]
Observe also that $a(f)W=qWa(f)$.

Applying these commutation relations to  $a(e_{i(l_1)})\cdots a(e_{i(1)})c_r(e_{j(1)})\cdots c_r(e_{j(s)})$, we can write $(*)$ as a finite linear combination of terms of the following form:
\[D_1\cdots D_{m_1}a(e_{i(l_1)})^{t}a(e)^{m_2}a_r(e)^{k_2-s}z_n,\]
where $m_1,m_2\geq 0,t\in\{0,1\}$ and $D_l\in \{c_r(e_{j(1)}), \cdots, c_r(e_{j(s)}), W\}, \forall 1\leq l\leq m_1$. 
If $t=1$, then we are back to Case 1; \\
If $t=0$, then one of the $D_k, 1\leq k\leq m_1$ must be $W$. But we know that $W$ will decrease the size of the vector in an exponential rate with respect to the length of its basic words, so as $n\to \omega$, the length of $z_n$ goes to infinity, thus 
\[\|D_1\cdots D_{m_1}a(e)^{m_2}a_r(e)^{k_2-s}z_n\|_2\to 0 \quad \text{as well.}\]
\end{proof}

A consequence of the theorem is a strengthening of maximal amenability:
\begin{thm}\label{AAP}
Let $-1< q<1$ be a real number with $|q|$ sufficiently small, then the inclusion $A\subset M$ of the generator masa inside the $q$-Gaussian von Neumann algebra, has the absorbing amenability property as introduced in \cite{BrothierWen15cup}(see also \cite{houdayer14gammastability},\cite{wen15radialmasa}). That is, for any diffuse subalgebra $C\subset A$, $A$ is the unique maximal amenable extension inside $M$.
\end{thm}
\begin{proof}
$A$ is shown to be mixing in $M$ by \cite{BikramMukherjee16generatormasa},\cite{wen2016singularityinq-gaussian}. Thus \cite[Theorem 8.1]{houdayer14structurebogoljubov} applies. One can alternatively use the argument in \cite[Proposition 1]{wen15radialmasa}.
\end{proof}

 As another application of the R{\u a}dulescu basis, we can give a very short proof of non-$\G$ for the $q$-Gaussian algebras, whenever Theorem \ref{thm: Radulescu basis} is true:

\begin{cor}[See also Avsec \cite{avsec2011s-solid}]
Let $\H_{\mathbb{R}}$ be a real Hilbert space with $\dim \H_{\mathbb{R}}\geq 2$. Let $-1<q<1$ be any real number such that Theorem \ref{thm: Radulescu basis} holds. Then $\G_q(\H_{\mathbb{R}})$ is a full factor.
\end{cor}
\begin{proof}
Let $e,f\in \H_{\mathbb{R}}$ be two orthogonal unit vectors and let $A=\G_q(\mathbb{R}e)$(resp. $B=\G_q(\mathbb{R}f)$) be the generators subalgebra associated with $e$(resp. $f$).  We construct as in the previous section the R{\u a}dulescu basis $\{\xi^i_{r,s}:i\in I, r,s \geq 0\}$ respect to $A$. Notice that $f^{\otimes n}\in T_n$, thus we may choose the basis such that for each $n\geq 1$, $\dfrac{f^{\otimes n}}{\sqrt{[n]_q!}}=\xi^{i_n}_{0,0}$ for some $i_n \in I$.

Suppose $x\in M'\cap M^{\omega}$ with $\tau(x)=0.$ Since $x\in A'\cap M^{\omega}$, by applying Theorem \ref{thm: essential support} for $A$ and by the choice of the basis, we can assume that $E_{B^{\omega}}(x)=0$. Choose a Haar unitary $u$ of $A$. Note that $ux=xu$ and $u\perp B$. Therefore by the s-AOP for $B\subset M$ as shown in Theorem \ref{s-AOP}, we have that $\tau(uxu^*x^*)=\tau(xx^*)=0.$
\end{proof}

\begin{rem}
It is already known that for all $-1<q<1$ and for all separable real Hilbert space $\H_{\mathbb{R}}$ with $\dim \H_{\mathbb{R}}\geq 2$,  $\G_q(\H_{\mathbb{R}})$ does not have property $\Gamma$: Avsec \cite{avsec2011s-solid} showed the strong-solidity for all $\G_q(\H_{\mathbb{R}})$ and it follows from an argument in \cite{ozawa04solidity} that solid factors do not have property $\Gamma$.
\end{rem}

\section*{Appendix: Strong asymptotic orthogonality property implies singularity}

\begin{prop}
Let $M$ be a finite von Neumann algebra and $A\subset M$ an abelian diffuse subalgebra. Assume that the inclusion $A\subset M$ has the strong asymptotic orthogonality property. Then $A$ is singular in $M$.
\end{prop}
\begin{proof}

\textbf{Claim 1:} $A$ is maximal abelian in $M$.\\
\textit{Proof of Claim 1}: Let $x\in M\ominus A \cap A'$. Then when viewed as constant sequences in $M^{\omega}$, we have 
\[x,x^*\in M^{\omega}\ominus A^{\omega}\cap A'.\]
Hence the definition of strong asymptotic orthogonality property implies that $\tau(x^*xx^*x)=0$, thus $x=0$.
\\
\\
Now let $w\in U(M)$ be a normalizer of $A$. Then the von Neumann subalgebra $N$ generated by $A$ and $w$ is amenable and $A$ is a Cartan subalgebra of $N$.
\\
\\
\textbf{Claim 2:} If $N\neq A$, there exists a non-zero partial isometry $v$ in N with the following properties:
\begin{itemize}
\item $v^*v=vv^*\in Z(N)$, where $Z(N)=N\cap N'$ is the center of $N$;
\item $\exists t\in \mathbb{N}$ such that $v^t=v^*v$;
\item $vAv^*=Av^*v;$
\item $v\perp A$. 
\end{itemize}
\textit{Proof of Claim 2}: Decompose $N$ as a direct sum
\[D=\bigoplus_{n=1}^{\infty}N_{I_n}\oplus N_{\text{II}},\]
where $N_{I_n}$ is of type $I_n$ and $N_{\text{II}}$ is of type II. Let $z_{I_n}, z_{II}$ be the corresponding central projections in $N$. Set $A_{I_n}:=Az_{I_n}, A_{II}:=Az_{II}$. Then $A_{I_n}\subset N_{I_n}$ and $A_{II}\subset N_{II}$ are Cartan subalgebras.
\\
\textbf{Case 1:} if for some $n\geq 2$, $z_{I_n}\neq 0$, then $A_{I_n}\subset N_{I_n}$ is of the form
\[Z(N_{I_n})\otimes \mathbb{C}^n\subset Z(N_{I_n})\otimes M_n(\mathbb{C}).\]
Hence $v=1_{Z(N_{I_n})}\otimes \begin{pmatrix}
0&1&\\
&&\ddots&\ddots&\\
&&&&&1\\
1&&&&&0
\end{pmatrix}$ does the job.
\\
\textbf{Case 2:} if $z_{II}\neq 0$, then we can assume that $A_{II}\subset N_{II}$ is of the form
\[Z(N_{II})\otimes A_0\otimes \mathbb{C}^2\subset Z(N_{II})\otimes R\otimes M_2(\mathbb{C}), \]
where $A_0\subset R$ is a Cartan subalgebra in the hyperfinite II$_1$ factor. Then $v=:1\otimes 1\otimes \begin{pmatrix}
0&1\\
1&0
\end{pmatrix}$ works.

Set $u=v+(1-vv^*)$, then it readily checks that $u$ is a normalizer of $A$ in $N$ and $u^t=1$. Thus $u$ defines a $\mathbb{Z}/t\mathbb{Z}$-action on $A$. Let $C:=A^{Ad(u)}$ be the fixed-point subalgebra. Then $C$ is diffuse and $v,v^*\in M\ominus A\cap C'$. Thus by the assumption on strong asymptotic orthogonality property of $A\subset M$, we have 
\[\tau(v^*vv^*v)=0,\]
thus $v=0$, a contradiction. Therefore, $A$ is singular in $M$.

 \end{proof}

\small
\bibliographystyle{plain}
\bibliography{ref}

\begin{thebibliography}{10}

\bibitem{avsec2011s-solid}
Stephen Avsec.
\newblock Strong solidity of the q-gaussian algebras for all {$-1< q< 1$}.
\newblock {\em arXiv:1110.4918}, 2011.

\bibitem{BikramMukherjee16generatormasa}
Panchugopal Bikram and Kunal Mukherjee.
\newblock Generator masas in {$q$}-deformed {A}raki-{W}oods von neumann
  algebras and factoriality.
\newblock {\em arXiv:1606.04752}, 2016.

\bibitem{boutonnet13hyperbolic}
R{\'e}mi Boutonnet and Alessandro Carderi.
\newblock Maximal amenable subalgebras of von neumann algebras associated with
  hyperbolic groups.
\newblock {\em arXiv}, (1310.5864), 2013.

\bibitem{BozejkoKummererSpeicher97q-gaussian}
Marek Bo{\.z}ejko, Burkhard K{\"u}mmerer, and Roland Speicher.
\newblock {$q$}-{G}aussian processes: non-commutative and classical aspects.
\newblock {\em Comm. Math. Phys.}, 185(1):129--154, 1997.

\bibitem{BozejkoSpeicher91brownian}
Marek Bo{\.z}ejko and Roland Speicher.
\newblock An example of a generalized {B}rownian motion.
\newblock {\em Comm. Math. Phys.}, 137(3):519--531, 1991.

\bibitem{BozejkoSpeicher94CPmaps}
Marek Bo{\.z}ejko and Roland Speicher.
\newblock Completely positive maps on {C}oxeter groups, deformed commutation
  relations, and operator spaces.
\newblock {\em Math. Ann.}, 300(1):97--120, 1994.

\bibitem{arnaud14maxinjective}
Arnaud Brothier.
\newblock The cup subalgebra of a {${\rm II}_1$} factor given by a subfactor
  planar algebra is maximal amenable.
\newblock {\em Pacific J. Math.}, 269(1):19--29, 2014.

\bibitem{BrothierWen15cup}
Arnaud Brothier and Chenxu Wen.
\newblock The cup subalgebra has the absorbing amenability property.
\newblock {\em International J. Math.}, 27(2), 2016.

\bibitem{CFM13mixing}
Jan Cameron, Junsheng Fang, and Kunal Mukherjee.
\newblock Mixing subalgebras of finite von {N}eumann algebras.
\newblock {\em New York J. Math.}, 19:343--366, 2013.

\bibitem{CFRW2010radialmasa}
Jan Cameron, Junsheng Fang, Mohan Ravichandran, and Stuart White.
\newblock The radial masa in a free group factor is maximal injective.
\newblock {\em J. London Math. Soc.}, 82(2):787--809, 2010.

\bibitem{dabrowki14freePDE}
Yoann Dabrowski.
\newblock A free stochastic partial differential equation.
\newblock {\em Ann. Inst. Henri Poincar\'e Probab. Stat.}, 50(4):1404--1455,
  2014.

\bibitem{primenessforfreegroupfactors}
Liming Ge.
\newblock Applications of free entropy to finite von neumann algebras, {II}.
\newblock {\em Ann. of Math. (2)}, 147(1):143--157, 1998.

\bibitem{GuionnetShlyakhtenko14freetransport}
A.~Guionnet and D.~Shlyakhtenko.
\newblock Free monotone transport.
\newblock {\em Invent. Math.}, 197(3):613--661, 2014.

\bibitem{hiai01q-arakiwoods}
Fumio Hiai.
\newblock {$q$}-deformed {A}raki-{W}oods algebras.
\newblock In {\em Operator algebras and mathematical physics ({C}onstan\c ta,
  2001)}, pages 169--202. Theta, Bucharest, 2003.

\bibitem{houdayer14exoticmaxinjective}
Cyril Houdayer.
\newblock A class of {II$_1$} factors with an exotic abelian maximal amenable
  subalgebra.
\newblock {\em Trans. Amer. Math. Soc.}, 366(7):3693--3707, 2014.

\bibitem{houdayer14structurebogoljubov}
Cyril Houdayer.
\newblock Structure of {II$_1$} factors arising from free bogoljubov actions of
  arbitrary groups.
\newblock {\em Adv. Math.}, 260:414--457, 2014.

\bibitem{houdayer14gammastability}
Cyril Houdayer.
\newblock Gamma stability in free product von neumann algebras.
\newblock {\em Commun. Math. Phys.}, 336(2):831--851, 2015.

\bibitem{nou04non-injectivity}
Alexandre Nou.
\newblock Non injectivity of the {$q$}-deformed von {N}eumann algebra.
\newblock {\em Math. Ann.}, 330(1):17--38, 2004.

\bibitem{ozawa04solidity}
Narutaka Ozawa.
\newblock Solid von {N}eumann algebras.
\newblock {\em Acta Math.}, 192(1):111--117, 2004.

\bibitem{op10uniquecartan}
Narutaka Ozawa and Sorin Popa.
\newblock On a class of {${\rm II}_1$} factors with at most one {C}artan
  subalgebra.
\newblock {\em Ann. of Math. (2)}, 172(1):713--749, 2010.

\bibitem{petersonthom11cocycle}
Jesse Peterson and Andreas Thom.
\newblock Group cocycles and the ring of affiliated operators.
\newblock {\em Invent. Math.}, 185(3):561--592, 2011.

\bibitem{popa83maxinjective}
Sorin Popa.
\newblock Maximal injective subalgebras in factors associated with free groups.
\newblock {\em Adv. Math.}, 50:27--48, 1983.

\bibitem{radulescu91radialmasa}
Florin R{\u{a}}dulescu.
\newblock Singularity of the radial subalgebra of {$\mathscr{L}(F_N)$} and the
  {P}uk{\' a}nszky invariant.
\newblock {\em Pacific J. Math.}, 151(2):297--306, 1991.

\bibitem{ricard2005factoriality}
{\'E}ric Ricard.
\newblock Factoriality of q-gaussian von neumann algebras.
\newblock {\em Comm. Math. Phys.}, 257(3):659--665, 2005.

\bibitem{shlyakhtenko97quasi-free}
Dimitri Shlyakhtenko.
\newblock Free quasi-free states.
\newblock {\em Pacific J. Math.}, 177(2):329--368, 1997.

\bibitem{freeentropyIII}
D.~Voiculescu.
\newblock The analogues of entropy and of fisher's information measure in free
  probability theory {III}: the absence of cartan subalgebras.
\newblock {\em Geom. and Funct. Anal.}, 6(1):172--199, 1996.

\bibitem{voiculescu92book}
D.~V. Voiculescu, K.~J. Dykema, and A.~Nica.
\newblock {\em Free random variables}, volume~1 of {\em CRM Monograph Series}.
\newblock American Mathematical Society, Providence, RI, 1992.

\bibitem{wen15radialmasa}
Chenxu Wen.
\newblock Maximal amenability and disjointness for the radial masa.
\newblock {\em J. Func. Anal.}, 270(2):787--801, 2016.

\bibitem{wen2016singularityinq-gaussian}
Chenxu Wen.
\newblock Singularity of the generator subalgebra in {$q$}-{G}aussian algebras.
\newblock {\em arXiv:1606.05420}, 2016.

\end{thebibliography}

\end{document}